\documentclass[letterpaper,10pt]{amsart}
\usepackage{amssymb}
\usepackage{amsthm}
\usepackage[utf8]{inputenc}
\usepackage[margin=1.0in]{geometry}
\usepackage{pst-node}
\usepackage{tikz-cd} 
\usepackage{thmtools}
\usepackage{thm-restate}
\usepackage{hyperref}
\usepackage{amsrefs}
\usepackage[noabbrev,capitalize]{cleveref}
\usepackage{amsmath}
\usepackage{amsfonts}
\usepackage{tikz-cd}
\usepackage{mathtools}
\usepackage{bbm}
\usepackage{hyperref}
\usepackage{blindtext, xcolor}
\usepackage{comment}
\usepackage{romannum}
\usepackage{multirow}
\usepackage{array}

\DeclareMathOperator{\ord}{ord}
\DeclareMathOperator{\Stab}{Stab}

\DeclareMathOperator{\Len}{Len}

\DeclareMathOperator{\perc}{\%}

\DeclareMathOperator{\lv}{\lvert}
\DeclareMathOperator{\rv}{\rvert}

\newcommand{\ZZ}{\mathbb{Z}}      
\newcommand{\RR}{\mathbb{R}}      

\newcommand{\mfT}{\mathfrak{T}}

\newcommand{\la}{\langle}    
\newcommand{\ra}{\rangle}

\newcommand{\cF}{\mathcal{F}}

\newcommand{\Mod}[1]{\ (\mathrm{mod}\ #1)}

\newtheorem{theorem}{Theorem}[section]

\newtheorem{lemma}[theorem]{Lemma}
\newtheorem{proposition}[theorem]{Proposition}
\newtheorem{corollary}[theorem]{Corollary}

\theoremstyle{definition}

\theoremstyle{remark}
\newtheorem{remark}[theorem]{Remark}

\newtheorem{definition}[theorem]{Definition}

\usepackage[english]{babel}
\usepackage[utf8]{inputenc}

\title{Sumsets of sequences in abelian groups and flags in field extensions}
\author{Sameera Vemulapalli}
\date{\today}

\usepackage[english]{babel}
\usepackage[utf8]{inputenc}

\begin{document}

\maketitle
\pagenumbering{arabic}

\begin{abstract}
For a finite abelian group $G$ with subsets $A$ and $B$, the sumset $AB$ is $\{ab \mid a\in A, b \in B\}$. A fundamental problem in additive combinatorics is to find a lower bound for the cardinality of $AB$ in terms of the cardinalities of $A$ and $B$. This article addresses the analogous problem for sequences in abelian groups and flags in field extensions. For a positive integer $n$, let $[n]$ denote the set $\{0,\dots,n-1\}$. To a finite abelian group $G$ of cardinality $n$ and an ordering  $G = \{1=v_0,\dots,v_{n-1}\}$, associate the function $T \colon [n] \times [n] \rightarrow [n]$ defined by
\[
	T(i,j) = \min\big\{k \in [n] \mid \{v_0,\dots,v_i\}\{v_0,\dots,v_j\} \subseteq \{v_0,\dots,v_k\}\big\}.
\]
Under the natural partial ordering, what functions $T$ are minimal as $\{1=v_0,\dots,v_{n-1}\}$ ranges across orderings of finite abelian groups of cardinality $n$? We also ask the analogous question for degree $n$ field extensions.

We explicitly classify all minimal $T$ when $n < 18$, $n$ is a prime power, or $n$ is a product of $2$ distinct primes. When $n$ is not as above, we explicitly construct orderings of abelian groups whose associated function $T$ is not contained in the above classification. We also associate to orderings a polyhedron encoding the data of $T$.
\end{abstract}

\section{Introduction}
Let $G$ be an abelian group, written multiplicatively, and let $A,B\subseteq G$ be finite subsets. The sumset $AB$ is $\{ab \; \mid \; a\in A, b \in B\}$. A fundamental problem in additive combinatorics is to find a lower bound for the cardinality of $AB$ in terms of the cardinalities of $A$ and $B$. For a finite set $S$, let $\lvert S \rvert$ denote its cardinality. Let $\Stab(AB) = \{g \in G \mid gAB = AB\}$. In 1961, Kneser obtained the following lower bound.

\begin{theorem}[Kneser \cite{kneser}]
\label{kneserorig}
For all abelian groups $G$ and finite subsets $A,B\subseteq G$, we have
\[
	\lvert AB \rvert \geq \lvert A \rvert + \lvert B \rvert - \lvert \Stab(AB) \rvert.
\]
\end{theorem}

If $G$ is a cyclic group of prime order $p$, then \cref{kneserorig} specializes to the Cauchy--Davenport theorem, which states that every pair of nonempty subsets $A, B\subseteq G$ satisfies $\lvert AB \rvert \geq \min\{p, \lvert A \rvert + \lvert B \rvert -1\}$. Bachoc, Serra, and Z\'emor generalized \cref{kneserorig} to field extensions. 

\begin{theorem}[Bachoc, Serra, Z\'emor \cite{zemor}, Theorem 2]
\label{zemorthm}
For all field extensions $L/K$ and nonempty $K$-subvector spaces $A,B \subseteq L$ of positive finite dimension, we have
\[
	\dim_KAB \geq \dim_KA + \dim_KB - \dim_K\Stab(AB).
\]
\end{theorem}

In this article, we instead study the sumsets of sequences in abelian groups and flags in field extensions.

\subsection{Setup}
For a positive integer $n$, let $[n]$ denote the set $\{0,\dots,n-1\}$.
\subsubsection{Flags of abelian groups}
Let $G$ be any finite abelian group, written multiplicatively, of cardinality $n$. A \emph{flag of $G$} is an indexed set $\cF = \{F_i\}_{i \in [n]}$ of subsets of $G$ such that
\[
	\{1\} = F_0 \subset F_1 \subset \dots \subset F_{n-1} = G
\]
and $\lvert F_i \rvert = i + 1$ for all $i \in [n]$. 

\subsubsection{Flags of field extensions}
Let $L/K$ be any degree $n$ field extension. A \emph{flag of $L$ over $K$} is an indexed set $\cF = \{F_i\}_{i \in [n]}$ of $K$-subvector spaces of $L$ such that
\[
	K = F_0 \subset F_1 \subset \dots \subset F_{n-1} = L
\]
and $\dim_K F_i = i + 1$ for all $i \in [n]$. 

\subsubsection{The function $T_{\cF}$}
Let $\cF = \{F_i\}_{i \in [n]}$ be a flag of an abelian group or a field extension. 
Associate to $\cF$ the function 
\begin{align*}
  T_{\cF} \colon [n]\times[n] & \longrightarrow [n] \\
  (i\;\;,\;\;j ) \; & \longmapsto \min\{k \in [n] \mid F_iF_j \subseteq F_k\}.
\end{align*}

We call the associated function $T_{\cF}$ the \emph{flag type} of $\cF$. Which functions $T_{\cF}$ arise from groups and field extensions? 

The function $T_{\cF}$ satisfies the following conditions: for all $i,j \in [n]$, we have
\begin{enumerate}
\item $T_{\cF}(i,j) = T(j,i)$;
\item $T_{\cF}(i,0) = i$;
\item and if $i < n-1$, we have $T_{\cF}(i,j) \leq T_{\cF}(i+1,j)$.
\end{enumerate}

Motivated by the functions $T_{\cF}$, we make the following definition.

\begin{restatable}{definition}{flagtypedefn}
\label{flagtypedefn}
A \emph{flag type of degree $n$} is a function $T \colon [n] \times [n] \rightarrow [n]$ such that for any $i,j \in [n]$, we have:
\begin{enumerate}
\item $T(i,j) = T(j,i)$;
\item $T(i,0) = i$;
\item and $T(i,j) \leq T(i+1,j)$ if $i < n-1$.
\end{enumerate}
\end{restatable}

\begin{definition}
We say a flag type is \emph{realizable} if it arises from a flag of an abelian group or field extension. 
\end{definition}

In this article, let $T$ denote a flag type. Is every flag type realizable? The answer is no: for any integers $i \geq 0$ and $j > 0$, let $(i \perc j)$ denote the remainder when dividing $i$ by $j$. \cref{overflowlemma} implies that if $T(i,j) < i + j$, then there exists an integer $1 < k < n$ such that $k \mid n$ and
\[
	(i \perc k) + (j \perc k) \neq (i + j) \perc k.
\]
There are also other constraints on flag types; see \cref{unrealizableflagtype} for more information.

 Morever, say $T$ is \emph{realizable for groups} if there exists a finite abelian group with a flag realizing $T$. Similarly, say $T$ is \emph{realizable for field extensions} if there exists a field extension with a flag realizing $T$. Every flag type that is realizable for groups is realizable for fields (\cref{groupfieldlemma}).

In this article, we aim to describe the set of realizable flag types. 

\subsection{Some fundamental flag types}
We now construct some flag types that play a fundamental role in what follows. 

\begin{definition}
\label{towerTypeDefn}
A \emph{tower type} is a $t$-tuple of integers $\mfT = (n_1,\dots,n_t) \in \ZZ_{>1}^t$ for some $t \geq 1$. We say $t$ is the \emph{length} of the tower type and $\prod_{i = 1}^tn_i$ is the \emph{degree} of the tower type.
\end{definition}

In this article, $\mfT$ will denote a tower type of length $t$ and degree $n$. Mixed radix notation will be useful to express conditions concisely.
\begin{definition}
\label{mixedradixdefn}
Choose a tower type $\mfT = (n_1,\dots,n_t)$ and $i \in [n]$. Writing $i$ in \emph{mixed radix notation with respect to $\mfT$} means writing
\[
	i = i_1 + i_2 n_1 + i_3 (n_1n_2) + \dots + i_t(n_1\dots n_{t-1})
\]
where $i_s$ is an integer such that $0 \leq i_s < n_s$ for $1\leq s \leq t$. Note that the integers $i_s$ are uniquely determined.
\end{definition}

Fix a tower type $\mfT = (n_1,\dots,n_t)$. For any positive integer $i$, let $C_i$ denote the cyclic group of order $i$, written multiplicatively. Define the finite abelian group $G = C_{n_1} \times \dots \times C_{n_t}$.

For $1 \leq i \leq t$, let $e_i \in G$ be a generator of the $i$-th component $C_{n_i}$. Construct a sequence $\{v_0,\dots,v_{n-1}\}$ elements of $G$ as follows. Write $i = i_1 + i_2 (n_1) + i_3 (n_1 n_2) + \dots + i_t(n_1\dots n_{t-1})$ in mixed radix notation with respect to $(n_1,\dots,n_t)$ and set $v_i \coloneqq e_1^{i_1}\dots e_t^{i_t}$.

The sequence $\{v_0,\dots,v_{n-1}\}$ is essentially lexicographic; it is
\[
	\bigg \{ 1, e_1, e_1^2,\dots,e_1^{n_1-1}, e_2, e_1e_2, e_1^2e_2,\dots,\prod_{s=1}^{t}e_s^{n_s-1}  \bigg \}.
\]

Define the flag $\cF(\mfT) = \{F_i\}_{i \in [n]}$ by $F_i = \{v_0,\dots,v_i\}$.

\begin{definition}
Let $T(\mfT)$ be the flag type of $\cF(\mfT)$.
\end{definition}

We may describe the flag types $T(n_1,\dots,n_t)$ even more explicitly; see \cref{explicitflagtype}. The flag types $T(n_1,\dots,n_t)$ arise from flags of abelian groups. We will show that, in an appropriate sense, they are often \emph{minimal} among various subsets of realizable flag types.

\subsection{Minimality and realizability of $T(\mfT)$}
Say a flag type $T$ is \emph{realizable for a finite abelian group $G$} if there exists a flag $\cF$ of $G$ such that $T = T_{\cF}$. The following theorem is proved in \cref{realizabilitysection}.
\begin{restatable}{theorem}{realizabilityforgroupsthm}
\label{realizabilityforgroupsthm}
For any finite abelian group $G$ and tower type $\mfT = (n_1,\dots,n_t)$, the flag type $T(\mfT)$ is realizable for $G$ if and only if $G$ has a filtration
\[
	\{1\} = G_0 \subset G_1 \subset \dots \subset G_t = G
\]
such that $G_i/G_{i-1}$ is cyclic of order $n_i$ for all $1\leq i \leq t$. 
\end{restatable}

For a field $K$, say $T$ is \emph{realizable over a field $K$} if there exists a field extension of $K$ with a flag $\cF$ such that $T = T_{\cF}$. Say $T$ is \emph{realizable for a field extension $L/K$} if there exists a flag $\cF$ of $L$ over $K$ such that $T = T_{\cF}$.  

The following theorem is proved in \cref{realizabilitysection}.
\begin{restatable}{theorem}{realizabilityforfieldsthm}
\label{realizabilityforfieldsthm}
For any degree $n$ field extension $L/K$ and tower type $\mfT = (n_1,\dots,n_t)$, the flag type $T(\mfT)$ is realizable for $L/K$ if and only if $L/K$ has a tower of subfields 
\[
	K = L_0 \subset L_1 \subset \dots \subset L_t = L
\]
where $n_i = [L_i \colon L_{i-1}]$ and the field extension $L_i/L_{i-1}$ is simple for all $1\leq i\leq t$.
\end{restatable}

We define a partial ordering $\leq$ on the set of flag types. For any two flag types $T$ and $T'$, say $T \leq T'$ if $T(i,j) \leq T'(i,j)$ for all $i,j \in [n]$. This partial ordering allows us to give a good notion of \emph{minimality} in many cases.

\begin{definition}
\label{minimalitydef}
We say:
\begin{enumerate}
\item a flag type is \emph{minimal} if it is minimal among realizable flag types;
\item a flag type $T$ is \emph{minimal for finite abelian groups} if $T$ is minimal among flag types that are realizable for finite abelian groups;
\item for a finite abelian group $G$, say $T$ is \emph{minimal for $G$} if $T$ is minimal among flag types that are realizable for $G$;
\item for a field $K$, a flag type $T$ is \emph{minimal over $K$} if $T$ is minimal amomng flag types that are realizable over $K$;
\item and for a degree $n$ field extension $L/K$, say $T$ is \emph{minimal for $L/K$} if $T$ is minimal among flag types that are realizable for $L/K$. 
\end{enumerate}
\end{definition}

\begin{definition}
We say a tower type $\mfT = (n_1,\dots,n_t)$ is \emph{prime} if $n_1,\dots,n_t$ are all prime.
\end{definition}

The following theorem is proved in \cref{realizabilitysection}.
\begin{restatable}{theorem}{minimality}
\label{minimality}
For a tower type $\mfT = (n_1,\dots,n_t)$, the flag type $T(\mfT)$ is minimal if and only if $\mfT$ is prime.
\end{restatable}

The following theorem in proved in \cref{completesect} and \cref{notcompletesect}.
\begin{restatable}{theorem}{minimalitycompletethm}
\label{minimalitycompletethm}
The set $\{T(\mfT) \; \mid \; \mfT \text{ is prime and has degree } n\}$ is the set of minimal flag types of degree $n$ if and only if
\begin{enumerate}
\item $n = p^k$, with $p$ prime and $k \geq 1$;
\item $n = pq$, with $p$ and $q$ distinct primes;
\item or $n = 12$.
\end{enumerate}
\end{restatable}

\subsection{Tower types of flags}

To a flag $\cF$ of a finite abelian group $G$, we can associate a tower of subgroups
\[
	\{1\} = G_0 \subset G_1 \subset \dots \subset G_t = G	
\]
such a subgroup is in the tower if and only if it is generated by $F_i$ for some $i \in [n]$. The \emph{tower type} of $\cF$ is
\[
	\bigg(\frac{\lv G_1 \rv}{\lv G_0 \rv},\dots, \frac{\lv G_t \rv}{\lv G_{t-1} \rv}\bigg)
\]
Similarly, to a flag $\cF$ of a field extension $L/K$, we can associate a tower of field extensions
\[
	K = L_0 \subset L_1 \subset \dots \subset L_t = L
\]
such a field is in the tower if and only if it is generated by $F_i$ for some $i \in [n]$. The \emph{tower type} of $\cF$ is
\[
	\big([L_1 \colon L_0],\dots, [L_t \colon L_{t-1}]\big).
\]

\begin{restatable}{theorem}{lenstrathm}
\label{lenstrathm}
Suppose $\mfT = (n_1,\dots,n_t)$ is as follows:
\begin{enumerate}
\item $n < 8$;
\item $n = 8$ and $\mfT \neq (8)$;
\item $n$ is prime;
\item $\mfT = (p,\dots,p)$ for a prime $p$;
\item $\mfT = (2,p)$ for a prime $p$;
\item or $\mfT = (3,p)$ for a prime $p$;
\end{enumerate}
Then $T(\mfT)$ is the unique flag type that is minimal among flags of tower type $\mfT$. 
\end{restatable}

\subsection{Polyhedra associated to flag types}
We will study flag types using polyhedra. Let $T$ be a flag type. 
\begin{restatable}{definition}{ptdefn}
\label{ptdefn}
Let $P_{T}$ be the set of $\mathbf{x} = (x_1,\dots,x_{n-1}) \in \RR^{n-1}$ such that $0 \leq x_1 \leq \dots \leq x_{n-1}$ and $x_{T(i,j)} \leq x_i + x_j$ for all $1 \leq i,j < n$.
\end{restatable}

The set of $P_T$ has a natural poset structure given by inclusion. The definition of $P_T$ implies that $P_T \subseteq P_{T'}$ if and only if $T' \leq T$. 

The following is proved in \cref{polysection}.
\begin{restatable}{lemma}{injectionlemma}
\label{injectionlemma}
The map from flag types to polyhedra sending $T$ to $P_T$ is an injection.
\end{restatable}

So, the poset of polyhedra $P_T$ is canonically dual to the poset of flag types.

\begin{restatable}{definition}{lenpdefn}
\label{lenpdefn}
Let $\Len(\mfT)$ be the polyhedron associated to the flag type $T(\mfT)$.
\end{restatable}

We have a variant of \cref{minimalitycompletethm} for polyhedra.
\begin{restatable}{theorem}{minimalitycompletethmbetter}
\label{minimalitycompletethmbetter}
We have 
\[
	\bigcup_T P_T =  \bigcup_{\mfT}\Len(\mfT)
\]
as $T$ ranges across realizable flag types if and only if $n$ is of the following form:
\begin{enumerate}
\item $n = p^k$, with $p$ prime and $k \geq 1$;
\item $n = pq$, with $p$ and $q$ distinct primes;
\item or $n = 12$.
\end{enumerate}
\end{restatable}

The following is a corollary of \cref{lenstrathm}.

\begin{restatable}{corollary}{lenstrathmcorol}
Suppose $\mfT = (n_1,\dots,n_t)$ is as follows:
\begin{enumerate}
\item $n < 8$;
\item $n = 8$ and $\mfT \neq (8)$;
\item $n$ is prime;
\item $\mfT = (p,\dots,p)$ for a prime $p$;
\item $\mfT = (2,p)$ for a prime $p$;
\item or $\mfT = (3,p)$ for a prime $p$.
\end{enumerate}
Then
\[
	\bigcup_{\cF} P_{T_{\cF}} = \Len(\mfT)
\]
as $\cF$ ranges across all flags with tower type $\mfT$.
\end{restatable}

We explicitly describe the polyhedron $P_T$ in terms of the flag type $T$. 

\begin{restatable}{definition}{cornerdef}
\label{cornerdef}
For $0 < i,j < n$, say $(i,j)$ is a \emph{corner} of $T$ if $T(i-1,j) < T(i,j)$ and $T(i,j-1) < T(i,j)$.
\end{restatable}

The corners of $T$ describe facets of $P_T$, as proved in \cref{polysection}.

\begin{restatable}{theorem}{cornerfacetthm}
\label{cornerfacetthm}
$P_T$ is an unbounded polyhedron of dimension $n-1$. For $i,j,k \in [n]$, the inequality $x_{k} \leq x_i + x_j$ defines a facet of $P_T$ if and only if $(i,j)$ is a corner and $T(i,j)=k$.
\end{restatable}

This explicit description of $P_T$ specializes to the polyhedron $P(n_1,\dots,n_t)$. To state how, we will first need the following definition.

\begin{restatable}{definition}{overflowdefn}
\label{overflowdefn}
Fix a tower type $\mfT=(n_1,\dots,n_t)$ and integers $0 \leq i, j \leq i + j < n$. We say the addition $i+j$ \emph{does not overflow modulo $\mfT$} if we write $i$, $j$, and $k=i+j$ in mixed radix notation with respect to $\mfT$ as 
\[
	i = i_1 + i_2 n_1 + i_3 (n_1n_2) + \dots + i_t(n_1\dots n_{t-1})
\]
\[
	j = j_1 + j_2 n_1 + i_3 (n_1n_2) + \dots + j_t(n_1\dots n_{t-1})
\]
\[
	k = k_1 + k_2 n_1 + k_3 (n_1n_2) + \dots + k_t(n_1\dots n_{t-1})
\]
and $i_s + j_s = k_s$ for all $1 \leq s \leq t$. Otherwise, we say the addition $i+j$ \emph{overflows modulo $\mfT$}.
\end{restatable}

The following is proved in \cref{polysection}.
\begin{restatable}{corollary}{explicitflagtype}
\label{explicitflagtype}
For any tower type $\mfT$ and $0 \leq i,j,i+j < n$, the following are equivalent:
\begin{enumerate}
\item $i + j$ does not overflow modulo $\mfT$;
\item $(i,j)$ is a corner of $\Len(\mfT)$ ;
\item and $T(\mfT)(i,j) = i+j$.
\end{enumerate}
\end{restatable}

All corners $(i,j)$ of $T(\mfT)$ satisfy $T(\mfT)(i,j) = i + j$. Is it true that $T(i,j) = i+j$ for any corner $(i,j)$ of a realizable flag type $T$?

We prove the following theorem in \cref{cornerstructuresec}.

\begin{restatable}{theorem}{cornersineq}
\label{cornersineq}
For a flag $\cF$ and every corner $(i,j)$ of $T_{\cF}$, we have $T_{\cF}(i + j) \geq i + j$.
\end{restatable}

The reverse inequality is not true in general; there exist realizable flag types $T$ with a corner $(i,j)$ such that $T(i + j) > i + j$. We prove the theorem below in \cref{cornerstructuresec}.

\begin{restatable}{theorem}{cornersstrongineq}
\label{cornersstrongineq}
There exists a minimal flag type $T$ of degree $18$ such that
\[
	P_T \setminus \big( P_{T(2,3,3)} \cup P_{T(3,2,3)} \cup P_{T(3,3,2)}\big) \neq \emptyset
\]
and T has a corner $(i,j)$ such that $T(i + j) > i + j$.
\end{restatable}

\subsection{Acknowledgments}
I am extremely grateful to Hendrik Lenstra for the many invaluable ideas, conversations, and corrections throughout the course of this project. I also thank Manjul Bhargava for suggesting the questions that led to this paper and for providing invaluable advice and encouragement throughout our research. Thank you as well to Jacob Tsimerman, Arul Shankar, Gilles Z\'emor, and Noga Alon for illuminating conversations. The author was supported by the NSF Graduate Research Fellowship.

\section{The polyhedron associated to a flag type $T$}
\label{polysection}
Let $T$ be any flag type of degree $n$. In this section, we will prove \cref{cornerfacetthm}.

\ptdefn*

\begin{proposition}
\label{dimprop}
The set $P_T$ is an unbounded polyhedron of dimension $n-1$.
\end{proposition}
\begin{proof}
Because $P_T$ is a nonempty (it contains the origin) and is the intersection of homogeneous linear inequalities, it is an unbounded polyhedron. We will show that $P_T$ has dimension $n-1$ by producing $n-1$ linearly independent points in $P_T$. For $i \in [n-1]$, define $\mathbf{x}^i = (x_1^i,\dots,x_{n-1}^i) \in \RR^{n-1}$ by
\begin{align*}
x_1^i = \dots = x_i^i &= 1 \\
x_{i+1}^i = \dots = x_{n-1}^i &= 2
\end{align*}
It is easy to verify that $\mathbf{x}^i \in P_T$. The $(n-1)\times (n-1)$ matrix whose rows consist of the $(n-1)$-tuples $x^0,\dots,x^{n-2}$ has rank $(n-1)$; this can be seen via elementary row and column operations.
\end{proof}

\cornerfacetthm*
\begin{proof}
The first claim follows from \cref{dimprop}. We now prove the second claim. Suppose $(i,j)$ is not a corner of $T$. Then, there exists a corner $(i',j')$ with $i' \leq i$, $j' \leq j$ and $T(i',j') = T(i,j)$. We will  show that $x_{T(i,j)} \leq x_i + x_j$ does not define a facet of $P_T$ by showing that $P_T \cap \{\mathbf{x} \in \RR^{n-1} \; \vert \; x_{T(i,j)} = x_{i'} + x_{j'}\}$ strictly contains $P_T \cap \{\mathbf{x} \in \RR^{n-1} \; \vert \; x_{T(i,j)} = x_{i} + x_{j}\}$. In other words, we will produce $\mathbf{x} \in P_T$ with $x_{T(i,j)} = x_{i'} + x_{j'} \neq x_i + x_j$.

If $i' \neq j'$, set $\mathbf{x} = (x_1,\dots,x_{n-1})$ where
\begin{align*}
x_1, \dots, x_{i'} &= 2 \\
x_{i'+1}, \dots, x_{j'} &= 3 \\
x_{j'+1}, \dots, x_{T(i,j) - 1} &= 4 \\
x_{T(i,j)}, \dots, x_{n-1} &= 5.
\end{align*}
If $i' = j'$, set $\mathbf{x} = (x_1,\dots,x_{n-1})$ where
\begin{align*}
x_1, \dots, x_{i'} &= 2 \\
x_{i'+1}, \dots, x_{T(i,j) - 1} &= 3 \\
x_{T(i,j)}, \dots, x_{n-1} &= 4.
\end{align*}

A computation verifies that $x_{T(i,j)} = x_{i'} + x_{j'} \neq x_i + x_j$,   and $\mathbf{x} \in P_{T_{\cF}}$.

Therefore, $P_T$ is the set of $\mathbf{x} = (x_1,\dots,x_{n-1}) \in \RR^{n-1}$ such that $0 \leq x_1 \leq \dots \leq x_{n-1}$, and $x_{T(i,j)} \leq x_i + x_j$ for all corners $(i,j)$ of $T$. Now, suppose $(i,j)$ is a corner of $T$. To show that $x_{T(i,j)} \leq x_i + x_j$ defines a facet of $P_T$, it suffices to produce a point $\mathbf{x} = (x_1,\dots,x_{n-1})\in \RR^{n-1}$ such that $0 \leq x_1 \leq \dots \leq x_{n-1}$, and $x_{T(i,j)} > x_i + x_j$ and $x_{T(i',j')} \leq x_{i'} + x_{j'}$ for any other corner $(i',j')$.

Set $\mathbf{x} = (x_1,\dots,x_{n-1})$ where
If $i \neq j$, set $\mathbf{x} = (x_1,\dots,x_{n-1})$ where
\begin{align*}
x_1, \dots, x_{i} &= 2 \\
x_{i+1}, \dots, x_{j} &= 3 \\
x_{j+1}, \dots, x_{T(i,j) - 1} &= 4 \\
x_{T(i,j)}, \dots, x_{n-1} &= 6.
\end{align*}
If $i = j$, set $\mathbf{x} = (x_1,\dots,x_{n-1})$ where
\begin{align*}
x_1, \dots, x_{i} &= 2 \\
x_{i+1}, \dots, x_{T(i,j) - 1} &= 3 \\
x_{T(i,j)}, \dots, x_{n-1} &= 5.
\end{align*}

A computation verifies that $\mathbf{x}$ satisfies the necessary conditions.
\end{proof}

We obtain the following lemma.
\injectionlemma*
\begin{proof}
A flag type is determined by its corners, and a polyhedra is determined by its facets. Apply \cref{cornerfacetthm} to complete the proof.
\end{proof}

\explicitflagtype*
\begin{proof}
This follows from the definition of $T(\mfT)$.
\end{proof}

We may compare flag types using corners.

\begin{lemma}
\label{testflagineqlemma}
For any two flag types $T$ and $T'$ such that $T \not \geq T'$, there exists a corner $(i,j)$ of $T'$ such that $T(i,j) < T'(i,j)$.
\end{lemma}
\begin{proof}
Because $T \not \geq T'$, we have $P_{T} \not \subseteq P_{T'}$. Thus, there exists integers $1 \leq i,j < n$ such that
\[
	P_{T'} \subset \{\mathbf{x} = (x_1,\dots,x_{n-1}) \in \RR^{n-1} \; \mid \; x_{T'(i,j)} \leq x_i + x_j\},
\]
\[
	P_{T'} \not\subset \{\mathbf{x} = (x_1,\dots,x_{n-1}) \in \RR^{n-1} \; \mid \; x_{T'(i,j)} \leq x_i + x_j\},
\]
and the linear inequality $ x_{T'(i,j)} \leq x_i + x_j$ defines a facet of $P_{T'}$. Thus \cref{cornerfacetthm}, there exists a corner $(i,j)$ of $T'$ such that $P_T \cap \{\mathbf{x} \in \RR^{n-1} \; \mid \; x_{T'(i,j)} > x_i + x_j\} \neq \emptyset$. Because $P_T \subseteq \{\mathbf{x} \in \RR^{n-1} \; \mid \; x_{T(i,j)} \leq x_i + x_j\}$, we have $T(i,j) < T'(i,j)$.
\end{proof}

\section{Realizability and minimality of $T(\mfT)$}
\label{realizabilitysection}
For every flag $\cF$ of a finite abelian group, there exists a sequence $\{v_0,\dots,v_{n-1}\}$ such that $F_i = \{v_0,\dots,v_i\}$. Similarly, for a flag $\cF$ of an field extension $L/K$, there exists a sequence $\{v_0,\dots,v_{n-1}\}$ such that $F_i = K\la v_0,\dots,v_i\ra$.

\begin{lemma}
If $T_{\cF}(i,j) = j$ then $F_i \subseteq \Stab(F_j)$.
\end{lemma}
\begin{proof}
If $T_{\cF}(i,j) = j$ then $F_j \subseteq F_iF_j \subseteq F_j$ so $F_iF_j = F_j$.
\end{proof}

\begin{lemma}
\label{cornerreadinglemma}
If $\cF$ is a flag of a finite abelian group and $(i,j)$ is a corner of $T_{\cF}$, then $v_iv_j = v_{T_{\cF}(i,j)}$. If $\cF$ is a flag of a field extension and $(i,j)$ is a corner of $T_{\cF}$, then $v_iv_j \in F_{T_{\cF}(i,j)}\setminus F_{T_{\cF}(i,j)-1}$.
\end{lemma}
\begin{proof}
Suppose $\cF$ is a flag of an abelian group. We have that $T_{\cF}(i,j) = \max\{k \in [n] \mid v_k \in F_iF_j\}$. Thus, $v_{T_{\cF}(i,j)} \in F_iF_j\setminus (F_{i-1}F_j \cup F_iF_{j-1}) = \{v_iv_j\}$ so $v_{T_{\cF}(i,j)} = v_iv_j$. The proof in the case of field extensions is entirely analogous.
\end{proof}

\realizabilityforgroupsthm*
\begin{proof}
First suppose $G$ has such a filtration and let $e_i \in G_i$ be a lift of a generator of $G_i/G_{i-1}$ for $i = 1,\dots,t$. Write $i = i_1 + i_2(n_1) + i_3(n_1n_2) + \dots + i_t(n_1\dots n_{t-1})$ in mixed radix notation, and define a sequence $\{v_0,\dots,v_{n-1}\}$ by $v_i = e_1^{i_1}\dots e_t^{i_t}$ for $i \in [n]$. Define a flag $\cF = \{F_i\}_{i\in [n]}$ by $F_i = \{v_0,\dots,v_i\}$. Observe that $T_{\cF}$ equals $T(\mfT)$.

Now suppose that $\cF$ is a flag such that $T_{\cF} = T(n_1,\dots,n_t)$, corresponding to a sequence $\{v_0,\dots,v_{n-1}\}$. Now, for all $i = 1,\dots,t$, we have $T_{\cF}(n_1\dots n_i - 1, n_1\dots n_i - 1) = n_1\dots n_i - 1$, so the $F_{n_1\dots n_i - 1}$ must be a group; call it $G_i$. So, we have a filtration
\[
	\{1\} = G_0 \subset G_1 \subset \dots \subset G_t = G
\]
such that $\lv G_i/G_{i-1} \rv = n_i$ for all $1 \leq i \leq t$. Now for $1 \leq i \leq t$ let $w_{i} = v_{n_1\dots n_{i-1}}$; we claim $w_i$ is a cyclic generator of $G_i/G_{i-1}$. If $n_i = 2$, this is trivial so assume $n_i \neq 2$.

For $1 \leq j < n_i$, the addition $(n_1\dots n_{i-1}) + (j-1)(n_1\dots n_{i-1})$ does not overflow modulo $(n_1,\dots,n_t)$ so $(n_1\dots n_{i-1}, (j-1)n_1\dots n_{i-1})$ is a corner of $T_{\cF}$ by \cref{cornerfacetthm}. Hence by \cref{cornerreadinglemma}, we see that
\[
	v_{jn_1\dots n_{i-1}} = v_{n_1\dots n_{i-1}}v_{(j-1)n_1\dots n_{i-1}}.
\]
By induction, 
\[
	v_{jn_1\dots n_{i-1}} = w_i^j.
\]
Therefore $w_i^j$ is distinct for $1 \leq j < n_i$ so $w_i$ is a cyclic generator of the group $G_i/G_{i-1}$ of cardinality $n_i$.
\end{proof}

\realizabilityforfieldsthm*
\begin{proof}
The proof is entirely analogous to the proof of \cref{realizabilityforgroupsthm}.
\end{proof}

\minimality*
\begin{proof}
Let $\mfT = (n_1,\dots,n_t)$ be a tower type where $n_i = m_1m_2$ for $m_1$ and $m_2$ positive integers. Then $\Len(\mfT)$ is strictly contained in $\Len(n_1,\dots,n_{i-1},m_1,m_2,n_{i+1},\dots,n_t)$ by \cref{cornerfacetthm}. Thus, 
\[
	T(n_1,\dots,n_{i-1},m_1,m_2,n_{i+1},\dots,n_t) \leq T(\mfT).
\]
So if $\mfT$ is not prime, the flag type $T(\mfT)$ not minimal.  

We will now prove that if $\mfT$ is prime, the flag type $T(\mfT)$ is minimal. We give the proof for finite abelian groups; the proof for field extensions is entirely analogous. Write $\mfT = (p_1,\dots,p_t)$.

Choose a flag $\cF$ of $G$ such that $T_{\cF} \leq T(\mfT)$. We will prove that $T_{\cF} = T(\mfT)$. Let $\{v_0,\dots,v_{n-1}\}$ be a sequence of elements of $G$ such that $F_i = \{v_0,\dots,v_i\}$. For all $1 \leq i \leq t$, we have 
\[
	p_1\dots p_i - 1 \leq T_{\cF}(p_1\dots p_i - 1, p_1\dots p_i - 1) \leq T(\mfT)(p_1\dots p_i - 1, p_1\dots p_i - 1) = p_1\dots p_i - 1.
\]
Thus, $p_1\dots p_i - 1 = T_{\cF}(p_1\dots p_i - 1, p_1\dots p_i - 1)$. Therefore the set $\{v_0,\dots,v_{p_1\dots p_i - 1}\}$ form a subgroup of $G$; call this subgroup $G_i$. 

Moreover, let $w_i = v_{n_1\dots n_{i-1}}$. Write $k \in [n]$ in mixed radix notation with respect to $(p_1,\dots,p_t)$ as 
\[
	k = k_1 + k_2(p_1) + \dots + k_t(p_1\dots p_t).
\]
We will inductively show that $v_k = w_1^{k_1}\dots w_t^{k_t}$, which implies that $T_{\cF} = T(\mfT)$. For $k = 0$ or $k = p_1\dots p_i$ for $1\leq i < t$, this is obvious from the definition. For any other $k \in [n]$ there exists a corner $(i,j)$ of $T(\mfT)$ with $T(\mfT)(i,j) = i+j = k$ by \cref{explicitflagtype}. Because $T_{\cF} \leq T(\mfT)$ and $\lv F_i F_j \rv = i+j$ by induction, we must have $v_k = v_iv_j$. 
\end{proof}

\section{Some useful lemmas}

The following lemmas will be necessary to prove our main theorems. Recall the following definition.

\overflowdefn*

\begin{definition}
For any integer $1 < m < n$ such that $m \mid n$, we say the addition $i + j$ \emph{overflows modulo $m$} if the addition $i + j$ overflows modulo $(m,n/m)$. Otherwise, we say the addition $i + j$ \emph{does not overflow modulo $m$}.
\end{definition}

Recall that for a subset $S$ of a finite abelian group, $\lv S \rv$ denotes the cardinality $S$. For a $K$-vector space $S$ in a field extension $L/K$,  let $\lv S \rv$ denote $\dim_K S$.

\begin{lemma}[\cite{kneser}, \cite{zemor} Theorem 2]
\label{overflowlemma}
Let $G$ be an abelian group of cardinality $n$. Suppose we have $i,j \in [n]$ and two subsets $I$ and $J$ of cardinality $i + 1$ and $j + 1$ respectively. Suppose $\lv IJ \rv \leq i+j$. Set $m = \lvert \Stab(F_i F_j) \rvert$ and write $i$ and $j$ in mixed radix notation with respect to $(m, n/m)$ as
\begin{align*}
i &= i_1 + i_2 m \\
j &= j_1 + j_2 m.
\end{align*}

Then $m > 1$, the addition $i + j$ overflows modulo $m$, 
\[
	\lv \Stab(IJ)I\rv = (i_2 + 1)m,
\]
\[
	\lv \Stab(IJ)J \rv = (j_2 + 1)m,
\]
\[
	\lv IJ \rv = (i_2 + j_2 + 1)m.
\]

Also, if $L/K$ is a degree $n$ field extension and $I,J \subseteq L$ are $K$-subvector spaces of dimension $i + 1$ and $j + 1$ respectively such that $\dim_K IJ \leq i+j$, then the same result holds, where here $m = \dim_k \Stab(IJ)$.  
\end{lemma}
\begin{proof}
We prove the lemma for groups. The proof for vector spaces is entirely analogous if one substitutes \cref{zemorthm} for \cref{kneserorig}. Observe that $IJ = (\Stab(IJ)I)(\Stab(IJ)J)$. Apply \cref{kneserorig} to obtain that
\begin{align*}
i_1 + i_2 m + j_1 + j_2 m &= i+j \\
&\geq \lv IJ \rv \\
&=  \lv (\Stab(IJ)I)(\Stab(IJ)J) \rv \\
&\geq \lv \Stab(IJ)I \rv + \lv \Stab(IJ)J \rv - m \\
&\geq (i_2 + 1)m + (j_2 + 1)m - m \\
&\geq i_2 m + j_2 m + m
\end{align*}
and thus $i_1 + j_1 \geq m$, so the addition $i + j$ overflows modulo $m$. Then $\Stab(IJ)I$, $\Stab(IJ)J$ and $IJ$ are unions of cosets of the group $\Stab(IJ)$. Now because $i_1 + j_1 < 2m$, we must have $\lv \Stab(IJ)I \rv = (i_2 + 1)m$, $\lv \Stab(IJ)J \rv = (j_2 + 1)m$, and $\lv IJ \rv = (i_2 + j_2 + 1)m$.
\end{proof}

\begin{proposition}
There exists flag types that are not realizable.
\label{unrealizableflagtype}
\end{proposition}
\begin{proof}
Let $n = 6$ and let $T$ be a flag type such that $T(1,1) = 1$ and $T(2,2) = 2$. It is not difficult to verify that such a flag type exists. Suppose $T$ is realized by a flag $\cF = \{F_i\}_{i \in [n]}$ of a finite abelian group; then $F_1$ must be a subgroup of order $2$ because $T(1,1) = 1$, and $F_3$ must be a subgroup of order $3$ because $T(2,2) = 2$. However, $F_2 \subseteq F_3$, which is a contradiction. An analogous argument shows that $T$ is not realized by a flag $\cF$ of a field extension.
\end{proof}

\begin{lemma}
\label{groupfieldlemma}
Let $G = C_{q_1}\times \dots \times C_{q_k}$ be the finite abelian group of order $n$ where $C_{q_i}$ is the cyclic group of order $q_i$ and $q_i$ is a prime power for all $i$. Let $L/K$ be a degree $n$ field extension containing subextensions $L_i/K$ such that the extension $L_i/K$ has a primitive element $\alpha_i$ with $\alpha_i^{q_i} \in K$, we have $\deg(L_i/K) = q_i$, and $\prod_i L_i = L$. Then every flag type realizable for $G$ is realizable for $L/K$.
\end{lemma}
\begin{proof}
Let $e_1$ be a generator of $C_{q_i}$. For any flag $\cF = \{F_i\}_{i \in [n]}$ of $G$, there is a sequence $\{1=v_0,\dots,v_{n-1}\}$ where $F_i = \{v_0,\dots,v_i\}$. For $i \in [n]$, let $i_1,\dots,i_k$ be integers such that $v_i = e_1^{i_1}\dots e_k^{i_k}$. Let $\alpha_i$ be a primitive element of $L_i/K$ such that $\alpha_i^{q_i} \in K$. Let $v_i' = \alpha_1^{i_1}\dots \alpha_k^{i_k}$ and let $\cF' = \{F'_i\}_{i \in [n]}$ be the flag of $L/K$ defined by $F_i' = K\la v_0',\dots,v_i'\ra$. Then a computation shows that $T_{\cF} = T_{\cF'}$.
\end{proof}

\section{Proof that the flag types $T(\mfT)$ for prime $\mfT$ form the set of minimal flag types for certain $n$}
\label{completesect}
Our purpose in the following two sections will be to prove the following two theorems.

\minimalitycompletethm*
\begin{proof}
Combine \cref{completethm}, \cref{notcompletethm}, and \cref{minimality}.
\end{proof}

\minimalitycompletethmbetter*
\begin{proof}
Combine \cref{completethm} and \cref{notcompletethm}.
\end{proof}

We prove \cref{completethm} in \cref{completesect} and we prove \cref{notcompletethm} in \cref{notcompletesect}.

For simplicity of exposition, we assume in this section that all flags are flags of abelian groups; the proofs for field extensions is entirely analogous. Simply replace the word ``cardinality'' with ``dimension'', the word ``group'' with ``field'', the word ``order'' (of an element in the group) with ``degree'' (of an element in the field extension), and the phrase ``union of cosets'' with ``vector space''. Moreover, replace any use of \cref{kneserorig} with \cref{zemorthm}.

\begin{theorem}
\label{completethm}
Suppose $n$ is of the following form:
\begin{enumerate}
\item $n = p^k$, with $p$ prime and $k \geq 1$;
\item $n = pq$, with $p$ and $q$ distinct primes;
\item or $n = 12$.
\end{enumerate}

Then, for every realizable flag type $T$, there exists a flag type $T(p_1,\dots,p_t)$ such that $T \geq T(p_1,\dots,p_t)$.
\end{theorem}
\begin{proof}
We separate into the following cases:
\begin{enumerate}
\item for $n = p^k$ for $p$ prime and $k \geq 1$, apply \cref{primepower};
\item for $n = 2p$ for $p$ an odd prime, apply \cref{2prefined};
\item for $n = pq$ for $p$ and $q$ distinct odd primes, apply \cref{largepqcase};
\item and for $n = 12$, apply \cref{12}.
\end{enumerate}
\end{proof}

The rest of this section is devoted to proving \cref{2prefined}, \cref{largepqcase}, \cref{primepower}, and \cref{12}.

\begin{remark}
For a tower type $\mfT = (n_1,\dots,n_t)$, we denote the tower type $T(\mfT) = T((n_1,\dots,n_t))$ by $T(n_1,\dots,n_t)$. When $t = 2$, we caution the reader to not confuse $T(n_1,n_2)$ with \emph{evaluating} some flag type $T$ at the numbers $n_1$ and $n_2$.
\end{remark}

\begin{proposition}
\label{2prefined}
Suppose $n = 2p$ for $p$ an odd prime. For every realizable flag type $T$, we have $T \geq T(2,p)$ or $T \geq T(p,2)$.
\end{proposition}
\begin{proof}
Assume for the sake of contradiction that there exists a flag $\cF$ such that $T_{\cF} \not \geq T(p,2)$ and $T_{\cF} \not \geq T(2,p)$. By \cref{testflagineqlemma}, there exists integers $0 < i_2 \leq j_2 < i_2 + j_2 < 2p$ such that $i_2 + j_2$ does not overflow modulo $(p,2)$ and $T_{\cF}(i_2 + j_2) < i_2 + j_2$. By \cref{overflowlemma} we must have that $F_{i_2} + F_{j_2}$ overflows modulo $2$. Moreover, 
\[
\lv\Stab(F_{i_2}F_{j_2})F_{i_2}\rv = \lv F_{i_2} \rv.
\]
Hence $F_{i_2}$ is a union of cosets of $\Stab(F_{i_2}F_{j_2})F_{i_2}$ and $\Stab(F_{i_2}F_{j_2})F_{i_2}$ is a group of order $2$.

Similarly, because $T_{\cF} \not \geq T(2,p)$, there must exist integers $0 < i_p \leq j_p < i_p + j_p < 2p$ such that $i_p + j_p$ does not overflow modulo $(2,p)$ and $T_{\cF}(i_p + j_p) < i_p + j_p$. Again, by \cref{overflowlemma} we must have that $i_p + j_p$ overflows modulo $p$ and $\Stab(F_{i_p}F_{j_p})$ is a group of order $p$. Because $i_p + j_p \geq p$, we have $j_p \geq \frac{p+1}{2}$. 

Because $F_{i_2}$ is a union of cosets of a group of order $2$, it cannot be contained in a group of order $p$. Therefore, we must have $F_{j_p}  \subseteq F_{i_2}$ and hence
\begin{align*}
i_2 + 1 &= \lv F_{i_2} \rv \\
&= \lv \Stab(F_{i_2}F_{j_2})F_{i_2} \rv \\
&\geq \lv \Stab(F_{i_2}F_{j_2})F_{j_p} \rv \\
&= 2\big(\frac{p+1}{2} + 1\big) = p+3.
\end{align*}
Thus $i_2 + j_2 \geq 2p$, which is a contradiction.
\end{proof}

\begin{proposition}
\label{largepqcase}
Suppose $n = pq$ for $p$ and $q$ distinct odd primes. Then for every realizable flag type $T$, we have $T \geq T(p,q)$ or $T \geq T(q,p)$.
\end{proposition}
\begin{proof}
Assume for the sake of contradiction that there exists a flag $\cF$ such that $T_{\cF} \not \geq T(p,q)$ and $T_{\cF} \not \geq T(q,p)$. By \cref{testflagineqlemma}, there exists integers $0 < i_q \leq j_q < i_q + j_q < pq$ such that $i_q + j_q$ does not overflow modulo $(p,q)$, but $T_{\cF}(i_q, j_q) > i_q + j_q$. By \cref{overflowlemma} we must have that $i_q + j_q$ overflows modulo $q$ and  $K_q = \Stab(F_{i_q} F_{j_q})$ is a group of order $q$. 

Similarly, because $T_{\cF} \not \geq T(q,p)$, there must exist integers $0 < i_p \leq j_p < i_p + j_p < pq$ such that $i_p + j_p$ does not overflow modulo $(q,p)$, but $T(i_p, j_p) > i_p + j_p$. By \cref{overflowlemma} we must have that $i_p + j_p$ overflows modulo $q$ and  $K_p = \Stab(F_{i_p}F_{j_p})$ is a group of order $p$.

If $i_q \leq i_p$ and $j_q \leq j_p$ then
\[
	K_q \subseteq F_{i_q} F_{j_q} \subseteq F_{i_p} F_{j_p},
\]
and because $F_{i_p}F_{j_p}$ is a union of cosets of $K_p$, we have 
\[
G = K_pK_q \subseteq F_{i_p}F_{j_p},
\]
which is a contradiction. Similarly, it is not possible for $i_p \leq i_q$ and $j_p \leq j_q$. Hence without loss of generality we may suppose $i_p \leq i_q \leq j_q \leq j_p$. Write $i_p$ and $j_p$ in mixed radix notation with respect to $(p,q)$ and write $i_q$ and $j_q$ in mixed radix notation with respect to $(q,p)$ as
\begin{align*}
i_p &= i_{1,p}p + i_{2,p} \\
j_p &= j_{1,p}p + j_{2,p} \\
i_q &= i_{1,q}q + i_{2,q} \\
j_q &= j_{1,q}q + j_{2,q}.
\end{align*} 
By \cref{overflowlemma}, we have that
\begin{align*}
\lv K_p F_{i_p}\rv &= \lv K_p\rv(i_{1,p}+1) \\
\lv K_p F_{j_p}\rv &= \lv K_p\rv(j_{1,p}+1) \\
\lv K_q F_{i_q}\rv &= \lv K_q\rv(i_{1,q}+1) \\
\lv K_q F_{j_q}\rv &= \lv K_q\rv(j_{1,q}+1).
\end{align*}

Note that
\[
i_q + 1 = \lv F_{i_q}\rv \leq \frac{\lv K_q F_{i_q}\rv}{\lv K_q \rv}\frac{\lv K_p F_{i_q}\rv}{\lv K_p \rv}\leq (i_{1,q}+1)\frac{\lv K_p F_{j_p}\rv}{\lv K_p \rv}= (i_{1,q}+1)(j_{1,p}+1)
\]
\[
j_q + 1= \lv F_{j_q}\rv \leq \frac{\lv K_q F_{j_q}\rv}{\lv K_q \rv}\frac{\lv K_p F_{j_q}\rv}{\lv K_p \rv} \leq (j_{1,q}+1)\frac{\lv K_p F_{j_p}\rv}{\lv K_p \rv}= (j_{1,q}+1)(j_{1,p}+1).
\]
Thus
\begin{align*}
i_q + j_q &< (i_{1,q}+1)(j_{1,p}+1) + (j_{1,q}+1)(j_{1,p}+1) \\
		  &= (i_{1,q}+j_{1,q}+2)(j_{1,p}+1) \\
		  &\leq p(j_{1,p}+1) \\
		  &\leq j_p.
\end{align*}
Because $i_q + j_q \leq j_p$, we have $K_q \subseteq F_{i_q + j_q - 1} \subseteq F_{j_p}$ and because $F_{i_p}F_{j_p}$ union of $K_p$-cosets, we have
\[
	K = K_qK_p \subseteq F_{i_p}F_{j_p},
\]
which is a contradiction.
\end{proof}

\begin{proposition}
\label{primepower}
Suppose $n = p^k$ for $p$ and odd prime and $k \geq 1$. Then for every realizable flag type $T$, we have $T \geq T(p,\dots,p)$.
\end{proposition}
\begin{proof}
Suppose there exists a realizable flag type $T$ such that $T_{\cF} \not \geq T(p,\dots,p)$. By \cref{testflagineqlemma}, there exists integers $0 < i \leq j < n$ such that $i + j$ does not overflow modulo $(p,\dots,p)$ and $T(i,j) < i+j$. By \cref{overflowlemma}, $i + j$ must overflow over some positive integer $m$ such that $m \mid p^k$, which is a contradiction.
\end{proof}

\begin{proposition}
\label{12}
Suppose $n = 12$. Then for every realizable flag type $T$, we have $T \geq T(3,2,2)$ or $T \geq T(2,3,2)$ or $T \geq T(2,2,3)$.
\end{proposition}
\begin{proof}
Assume for the sake of contradiction that there exists a flag $\cF$ such that $T_{\cF} \not \geq T(3,2,2)$, $T_{\cF} \not \geq T(2,3,2)$, and $T_{\cF} \not \geq T(2,2,3)$. By \cref{testflagineqlemma}, there exists $i_1$, $i_2$, $i_3$, $j_1$, $j_2$, $j_3$ such that
\begin{align*}
0 < i_1 \leq j_1 &< i_1 + j_1 < 12 \\
0 < i_2 \leq j_2 &< i_2 + j_2 < 12 \\
0 < i_3 \leq j_3 &< i_3 + j_3 < 12,
\end{align*}
and $i_1 + j_1$ (resp. $i_2 + j_2$, $i_3 + j_3$) does not overflow modulo $(3,2,2)$ (resp. $(2,3,2)$, $(2,2,3)$), and
\begin{align*}
\label{cond}
T_{\cF}(i_1 + j_1) &> i_1 + j_1 \\
T_{\cF}(i_2 + j_2) &> i_2 + j_2 \\
T_{\cF}(i_3 + j_3) &> i_3 + j_3.
\end{align*}

By \cref{overflowlemma} we have that  
\[
	K_1 = \Stab(F_{i_1}F_{j_1})
\]
\[
	K_2 = \Stab(F_{i_2}F_{j_2})
\]
\[
	K_3 = \Stab(F_{i_3} F_{j_3})
\]
are a nontrivial proper subgroups; let $m_1 = \lvert K_1 \rv$, let $m_2 = \lv K_2 \rv$, and let $m_3 = \lv K_3 \rv$. 

The list of possible triples $(i_1,j_1,m_1)$ for which $0 < i_1 \leq j_1 < i_1 + j_1 < 12$, the addition $i_1 + j_1$ overflows modulo $m_1$, we have $1 < m_1 < 12$, we have $m_1 \mid 12$, and the addition $i_1 + j_1$ does not overflow modulo $(3,2,2)$ is
\[
	\{(1,1,2), (1,3,2), (1,3,4), (1,7,2), (1,7,4), (1,9,2), (2,3,4), (2,6,4), (3,6,4), (3,7,2), (3,7,4)\}.
\]
The list of possible triples $(i_2,j_2,m_2)$ for which $0 < i_2 \leq j_2 < i_2 + j_2 < 12$, the addition $i_2 + j_2$ overflows modulo $m_2$, we have $1 < m_2 < 12$, we have $m_2 \mid 12$, and the addition $i_2 + j_2$ does not overflow modulo $(2,3,2)$ is
\[
	\{(1,2,3), (1,8,3), (2,2,3), (2,2,4), (2,3,4), (2,6,4), (2,7,3), (2,7,4), (2,8,3), (3,6,4)\}.
\]
The list of possible triples $(i_3,j_3,m_3)$ for which $0 < i_3 \leq j_3 < i_3 + j_3 < 12$, the addition $i_3 + j_3$ overflows modulo $m_3$, we have $1 < m_3 < 12$, we have $m_3 \mid 12$, and the addition $i_3 + j_3$ does not overflow modulo $(2,2,3)$ is
\[
	\{(1,2,3), (1,8,3), (2,4,3), (2,4,6), (2,5,3), (2,5,6), (2,8,3), (3,4,6), (4,4,6), (4,5,3), (4,5,6)\}.
\]
We will show that every combination of integers $(i_1,j_1,m_1)$, $(i_2,j_2,m_2)$, and $(i_3,j_3,m_3)$ from the lists above give rise to a contradiction.

Then because $i_1 < m_1$, by \cref{overflowlemma}, we have $K_1F_{i_1} = K_1$ so $F_{i_1} \subseteq K_{1}$. Write $F_1 = \{1, v_1\}$. Then $\ord(v_1) \mid m_1 \mid 4$. 

If $\ord(v_1) = 4$ then as $\ord(v_1) \mid m_1 \mid 4$, we have that $m_1 = 4$. If $i_3 < m_3$ then $F_{i_3} \subseteq K_{3}$ and thus $\deg(v_1) \mid m_{3} \mid 6$, which is a contradiction. Hence $i_3 > m_3$ so we must have $m_3 = 3$ and $i_3 = 4$ and $j_3 = 5$. Moreover, by \cref{overflowlemma}, $\lv K_{3}F_5\rv = 2 \lv K_3\rv$. Because $m_1=4$ and $m_3=3$ are coprime, we have $F_{i_1} \ra \subseteq K_{1}$ and 
\[
	i_1+1 = \lv F_{i_1} \rv = \frac{\lv K_3 F_{i_1}\rv}{\lv K_3 \rv} \leq \frac{\lv K_3 F_{5}\rv}{\lv K_3 \rv} = 2,
\]
which is a contradiction.

Thus $\ord(v_1) = 2$, and therefore without loss of generality we may suppose $i_1 = j_1 = 1$ and $m_1 = 2$, as  $\ord(v_1) = 2$ implies that
\[
	F_1F_1 = F_1.
\]

Notice that $i_2 < m_2$, so $F_{i_2} \subseteq K_{2}$ and thus $\ord(v_1) = 2 \mid \ord(K_2)$; hence $m_2 = 4$. 

If $m_3 = 3$, then if $i_3 < m_3$ we have $F_{i_3} \ra \subseteq K_{3}$ and thus $\ord(v_1) \mid K_{3} \mid 3$, which is a contradiction. Thus $i_3 > m_3$, so $i_3 = 4$ and $j_3 = 5$. By \cref{overflowlemma}, $\lv K_{3}F_5\rv = 2\lv K_3\rv$; thus 
\[
	F_5 = K_{3}F_1
\]
is a group of order $6$. However, 
\begin{align*}
 K_{3}F_1 &= \Stab(F_5) \\
 &\subseteq \Stab(F_4F_5) \\
 &= K_3,
\end{align*}
which is a contradiction.

Hence, $m_3 = 6$. Because $i_2 < m_2$ and $i_3 < m_3$ and $2 \leq i_2,i_3$, we have $F_2 \subseteq K_{2} \cap K_{3}$, which is a contradiction, because
\[
\lv K_{2} \cap K_{3} \rv \leq \gcd(m_2,m_3) = \gcd(4,6) = 2.
\]
\end{proof}

\section{Proof that the flag types $T(p_1,\dots,p_t)$ do not form the set of minimal flag types for certain $n$}
\label{notcompletesect}

\begin{theorem}
\label{notcompletethm}
Suppose $n$ is not of the following form:
\begin{enumerate}
\item $n = p^k$, with $p$ prime and $k \geq 1$;
\item $n = pq$, with $p$ and $q$ distinct primes;
\item or $n = 12$.
\end{enumerate}

Then there exists a realizable flag type $T$ such that 
\[
	P_T \not\subseteq \bigcup_{(p_1,\dots,p_t)}P_{T(p_1,\dots,p_t)}.
\]
\end{theorem}
\begin{proof}
By \cref{refconjcounterexextensionlemma}, proving the statement of \cref{notcompletethm} for degree $n$ implies the statement of \cref{notcompletethm} for degree $nt$ for any positive integer $t$. Therefore, it suffices to prove the theorem when:
\begin{enumerate}
\item $n = p^2q$ for two distinct odd primes $p$ and $q$ with $p < q$, in which case \cref{p2q} provides a proof;
\item $n = pqr$ for three primes $p$, $q$, and $r$ with $p < q \leq r$, in which case \cref{pqr} provides a proof;
\item or $n = 4p$ for a prime $p \neq 2,3$, in which case \cref{4p} provides a proof.
\end{enumerate}
\end{proof}

The rest of this section is dedicated to proving \cref{refconjcounterexextensionlemma}, \cref{p2q}, \cref{pqr}, and \cref{4p}.

\begin{lemma}
\label{refconjcounterexextensionlemma}
Let $n,m \in \ZZ_{>1}$ be such that $m \mid n$.
If there exists a realizable flag type $T$ of degree $m$ such that 
\[
	P_T \not\subseteq \bigcup_{(p_1,\dots,p_t)}P_{T(p_1,\dots,p_t)}
\]
where $(p_1,\dots,p_t)$ range across all tuples with prime entries such that $p_1\dots p_t = m$, then there exists a realizable flag type $T'$ of degree $n$ such that 
\[
	P_{T'} \not\subseteq \bigcup_{(p'_1,\dots,p'_t)}P_{T(p'_1,\dots,p'_t)}
\]
where $(p'_1,\dots,p'_t)$ range across all tuples with prime entries such that $p'_1\dots p'_t = n$.
\end{lemma}
\begin{proof}
The polyhedron $P_T$ has dimension $m-1$. Therefore, the set
\[
	S = \bigg(P_T \setminus \bigcup_{(p_1,\dots,p_t)}P_{T(p_1,\dots,p_t)}\bigg) \cap \{\mathbf{x} \in \RR^{m-1} \; \vert \; x_k \neq x_{i} + x_j \; \forall  \; 1 \leq i,j,k < m \}
\] 
is nonempty.

Let $G$ be a finite abelian group of cardinality $m$ with a flag $\cF$ realizing $T$. Let $\{v_0,\dots,v_{m-1}\}$ be a sequence of elements in $G$ such that $F_i = \{v_0,\dots,v_{i}\}$.

Let $G' = G \times C_p$, where $C_p$ is a cyclic group of order $p$. Let $e$ be a generator of $C_p$. Define the sequence $\{1 = v'_0,\dots,v'_{n-1} \}$  by $v'_i = v_{i_2}e^{i_1}$ for $i = i_1 + i_2 p$ in mixed radix notation with respect to $(p,m)$. Define a flag $\cF' = \{F'_i\}_{i \in [n]}$ of $G'$ by $F'_i = \{v'_0,\dots,v'_i\}$.

Choose $\mathbf{x} = (x_1,\dots,x_{m-1}) \in S$. Let
\[
	\epsilon = \min_{1 \leq i < m-1}\{x_{i+1} - x_i\}
\]
and note that $\epsilon \neq 0$. Define the point $\mathbf{x}' = (x_1',\dots,x_{n-1}') \in \RR^{n-1}$ by
\[
	x_i' = \epsilon \frac{i_1}{2p} + x_{i_2}.
\]

We will show that $\mathbf{x}' \in P_{T'}$. We first show that $0\leq x'_1\leq \dots\leq x'_{mp-1}$. For $1 \leq i < n-1$, write $i = i_1 + i_2 p$ in mixed radix notation with respect to $(p,m)$. If $i_1 \neq p-1$, then
\[
	x_{i+1}' = \epsilon \frac{i_1 + 1}{2p} + x_{i_2} \geq  \epsilon \frac{i_1}{2p} + x_{i_2} = x_i'.
\]
If $i_1 = p-1$ then 
\[
	x_{i+1}' = x_{i_2+1} \geq  \epsilon + x_{i_2} \geq \epsilon \frac{i_1}{2p} + x_{i_2} = x_i'.
\]

For any $1 \leq i,j,k < n$, we now show that $x'_{T_{\cF'}(i,j)} \leq x'_i + x'_j$. Let $k = T_{\cF'}(i,j)$. Write
\begin{align*}
i &= i_1 + i_2 p \\
j &= j_1 + j_2 p \\
k &= k_1 + k_2 p
\end{align*}
in mixed radix notation with respect to $(p,m)$. Then we have $k_1 \leq i_1+j_1$  so $x_{k_2}\leq x_{i_2}+x_{j_2}$. Then
\begin{align*}
x'_k &= \epsilon \frac{k_1}{2p} + x_{k_2} \\
&\leq \epsilon \frac{k_1}{2p} + x_{k_2}\\
&\leq \epsilon\bigg(\frac{i_1}{2p} + \frac{j_1}{2p}\bigg) + x_{i_2} + x_{j_2} \\
&\leq \epsilon \frac{i_1}{2p} + x_{i_2} + \epsilon \frac{j_1}{2p} + x_{j_2} \\
&\leq x'_i + x'_j.
\end{align*}

We will now prove that
\[
	\mathbf{x}' \notin \bigcup_{(p'_1,\dots,p'_t)}P_{T(p'_1,\dots,p'_t)}
\]
where $(p_1,\dots,p_t)$ ranges across all tuples with prime entires such that $\prod_i p'_i = n$. Fix such a tuple $(p'_1,\dots,p'_t)$. If $p'_1 \neq p$ then
\[
	x_1'+x_{p-1}' = \frac{\epsilon}{2} < x_1 = x'_p,
\]
so $\mathbf{x}' \notin P_{T(p_1,\dots,p_t)}$ because $1 + (p-1)$ does not overflow modulo $(p_1,\dots,p_t)$ and $x_p' > x'_1 + x'_{p-1}$. If $p'_1 = p$ then $p'_2\dots p'_t = m$. Because $\mathbf{x} \notin P_{T(p_2',\dots,p_t')}$, we can choose $1 \leq i \leq j < i+j < m$ such that $i+j$ does not overflow modulo $(p'_2,\dots,p'_t)$ and $x_{i+j} > x_i + x_j$. Note that $pi + pj$ does not overflow modulo $(p'_1,\dots,p'_t)$ and 
\[
	x_{pi+pj}' = x_{i+j} > x_{i} + x_{j} = x_{pi}' + x_{pj}'.
\]
Thus $\mathbf{x}' \notin P_{T(p'_1,\dots,p'_t)}$.
\end{proof}

\begin{proposition}
\label{p2q}
Let $p$ and $q$ be two distinct odd primes with $p < q$ and let $n = p^2q$. Then there exists a realizable flag type $T$ such that 
\[
	P_T \not \subseteq P_{T(p,p,q)} \cup P_{T(p,q,p)} \cup P_{T(q,p,p)}.
\]
\end{proposition}
\begin{proof}

Let $G = C_q \times C_p \times C_q$ and let $e_i$ be a generator of the $i$-th component of $G$.

Define a sequence $\{1=v_0,\dots,v_{n - 1}\}$ as follows. For $0 \leq i < q$, set $v_i = e_1^{i}$. For $q \leq i < n$ and $1 \leq i' < n$, write $i' = i'_1 + i'_2p + i_3'pq$ in mixed radix notation with respect to $(p,q,p)$. 
Inductively define $v_i$ as follows. Choose $i'$ minimal such that $e_2^{i'_1}e_1^{i'_2}e_3^{i'_3} \notin \{v_0,\dots,v_{i-1}\}$, and set $v_i = e_2^{i'_1}e_1^{i'_2}e_3^{i'_3}$. Observe that for $i \geq pq$, $e_2^{i'_1}e_1^{i'_2}e_3^{i'_3} = e_2^{i_1}e_1^{i_2}e_3^{i_3}$.
Define a flag $\cF = \{F_i\}_{i \in [n]}$ by $F_i = \{v_0,\dots,v_i\}$.

Because
\begin{align*}
F_1  &= \{ 1, e_1 \} \\
F_{q-1} &= \la e_1 \ra \\
F_{pq+1} &= \la e_1,e_2 \ra \cup \{ e_3, e_2e_3 \} \\
F_{pq+p-1} &= \la e_1,e_2 \ra \cup \la e_2 \ra\{ e_3\} = \la e_2\ra( \la e_1\ra \cup \{e_3\}) \\
F_{2pq+p-1} &= \la e_1, e_2\ra \{1, e_3\} \cup \la e_2\ra \{e_3^2\},
\end{align*}
we have 
\[
F_1 F_{q-1} = F_{q-1}
\]
and
\begin{align*}
F_{pq+1}F_{pq+p-1} &= (\la e_1, e_2 \ra \cup \{ e_3, e_2e_3 \})(\la e_2 \ra (\la e_1 \ra \cup \{e_3\}))\\
&= \la e_1, e_2\ra \{1, e_3\} \cup \la e_2 \ra\{ e_3^2\} \\
&= F_{2pq+p-1}.
\end{align*}

We claim there exists $\mathbf{x} \in P_{T_{\cF}}$ such that $x_q > x_1 + x_{q-1}$ and $x_{2pq+p} > x_{pq + 1} + x_{pq + p - 1}$. In particular, for $1 \leq i < q$ set $x_i = \frac{i}{2q}$. For $q \leq i < pq$ set $x_i = 1$. For $pq \leq i < p^2q$ write $i = i_1 + i_2p + i_3pq$ in mixed radix notation with respect to $(p,q,p)$ and set $x_i = i_1 \frac{1}{4pq} + i_2 \frac{1}{2q} + i_3$. 

We first show that $0 \leq x_1 \leq \dots x_{n-1}$. If $1\leq i < q-1$, then
\[
	x_{i+1} = \frac{i+1}{2q} \geq \frac{i}{2q} = x_i.
\]
Note also that
\[
	x_q = 1 > \frac{q-1}{2q} = x_{q-1}
\]
If $q\leq i < pq-1$, 
\[
	x_{i+1} = 1 = x_i.
\]
If $pq\leq i < n-1$ then write $i+1 = (i+1)_1 + (i+1)_2p + (i+1)_3pq$ in mixed radix notation with respect to $(p,q,p)$ as well. Then if $i_1 = p-1$ and $i_2 = q-1$ then $(i+1)_1 = 0$ and $(i+1)_2 = 0$ and $(i+1)_3 = i_3 + 1$. Then
\begin{align*}
x_{i+1} &= (i+1)_1 \frac{1}{4pq} + (i+1)_2 \frac{1}{2q} + (i+1)_3 \\
&= (i_3 + 1) \\
& > i_1 \frac{1}{4pq} + i_2 \frac{1}{2q} + i_3 \\
&= x_i.
\end{align*}
Instead if $i_1 = p-1$ and $i_2 \neq q-1$ then $(i+1)_1 = 0$ and $(i+1)_2 = i_2 + 1$ and $(i+1)_3 = i_3$. Then
\begin{align*}
x_{i+1} &= (i+1)_1 \frac{1}{4pq} + (i+1)_2 \frac{1}{2q} + (i+1)_3 \\
&= (i_2+1) \frac{1}{2q} + i_3 \\
&> i_1 \frac{1}{4pq} + i_2 \frac{1}{2q} + i_3 \\
&= x_i.
\end{align*}
Finally, if $i_1 \neq p-1$ then $(i+1)_1 = i_1+1$ and $(i+1)_2 = i_2$ and $(i+1)_3 = i_3$. Then
\begin{align*}
x_{i+1} &= (i+1)_1 \frac{1}{4pq} + (i+1)_2 \frac{1}{2q} + (i+1)_3 \\
&= (i_1+1) \frac{1}{4pq} + i_2\frac{1}{2q} + i_3 \\
&> i_1 \frac{1}{4pq} + i_2 \frac{1}{2q} + i_3 \\
&= x_i.
\end{align*}
Therefore, $0 \leq x_1 \leq \dots x_{n-1}$. 

We now show that for all $1\leq i,j < n$, we have $x_{T(i,j)} \leq x_i + x_j$. Let $k = T(i,j)$.

\noindent \textbf{Case 1: $1 \leq i < q$ and $1 \leq j < q$}.
Then $k = \min(q-1, i+j)$. Thus
\[
	x_k = \frac{k}{2q} \leq \frac{i}{2q} + \frac{j}{2q} = x_i + x_j.
\]
\noindent \textbf{Case 2: $1 \leq i < q$ and $q \leq j < pq$}.
Then as $F_{pq-1} = \la e_1,e_2 \ra$, we have $j < k < pq$. Thus
\[
	x_k = 1 \leq \frac{i}{2q} + 1 = x_i + x_j.
\]
\noindent \textbf{Case 3: $1 \leq i < q$ and $pq \leq j < n$}.
Write $j = j_1 + j_2p + j_3pq$ in mixed radix notation with respect to $(p,q,p)$. Then $j < k \leq j_1 + \min(q-1,i+j_2)p + j_3pq$. Then
\begin{align*}
x_k &\leq x_{j_1 + \min(q-1,i+j_2)p + j_3pq} \\
&= j_1 \frac{1}{4pq} + \min(q-1,i+j_2) \frac{1}{2q} + j_3 \\
&\leq \frac{i}{2q} + j_1 \frac{1}{4pq} + j_2 \frac{1}{2q} + j_3 \\
&= x_i + x_j.
\end{align*}

\noindent \textbf{Case $4$: $q \leq i < pq$ and $q \leq j < pq$}.
Then as $F_{pq-1} = \la e_1,e_2 \ra$, we have $j < k < pq$. Thus
\[
	x_k = 1 \leq \frac{i}{2q} + 1 = x_i + x_j.
\]
\noindent \textbf{Case $4$: $q \leq i < pq$ and $pq \leq j < n$}.
Write $j = j_1 + j_2p + j_3pq$ in mixed radix notation with respect to $(p,q,p)$. Because $F_{pq-1} = \la e_1,e_2\ra$, then $j < k \leq (p-1) + (q-1)p + j_3pq$. Then
\begin{align*}
x_k &\leq x_{(p-1) + (q-1)p + j_3pq} \\
& = (p-1) \frac{1}{4pq} + (q-1) \frac{1}{2q} + j_3 \\
& \leq 1 + j_3 \\
&= 1 + j_1 \frac{1}{4pq} + j_2 \frac{1}{2q} + j_3  \\
&= x_i + x_j.
\end{align*}

\noindent \textbf{Case 5: $pq \leq i < n$}.
Write
\begin{align*}
i &= i_1 + i_2p + i_3pq \\
j &= j_1 + j_2p + j_3pq
\end{align*}
in mixed radix notation with respect to $(p,q,p)$. Recall that for $i \geq pq$, we have $v_i = e_2^{i_1}e_1^{i_2}e_3^{i_3}$. Thus, $j < k \leq \min(p-1,i_1+j_1) + \min(q-1,i_2+j_2)p + (i_3 + j_3)pq$.  Then
\begin{align*}
x_k &\leq x_{\min(p-1,i_1+j_1) + \min(q-1,i_2+j_2)p + (i_3 + j_3)pq} \\
& = \min(p-1,i_1+j_1) \frac{1}{4pq} + \min(q-1,i_2+j_2)\frac{1}{2q} + (i_3 + j_3) \\
&\leq i_1 \frac{1}{4pq} + i_2 \frac{1}{2q} + i_3 + j_1 \frac{1}{4pq} + j_2 \frac{1}{2q} + j_3 \\
&=x_i + x_j.
\end{align*}

Thus, for all $1 \leq i,j < n$, we have if $x_{T(i,j)} \leq x_i + x_j$. Moreover, we have that
\[
	x_q = 1 > \frac{1}{2q} + \frac{q-1}{2q} = x_1 + x_{q-1}
\]
and
\[
	x_{2pq+p} = \frac{1}{2q} + 2 > \bigg (\frac{1}{4pq} + 1 \bigg) + \bigg (\frac{p-1}{4pq} + 1 \bigg) = x_{pq+1} + x_{pq + p - 1}.
\]
Moreover 
\[
	P_{T(p,p,q)} \cup P_{T(p,q,p)} \subseteq \{\mathbf{x} \in \RR^{p^2q-1} \mid x_{q} \leq x_1 + x_{q-1} \}
\]
and
\[
	P_{T(q,p,p)} \subseteq \{\mathbf{x} \in \RR^{p^2q-1} \mid x_{2pq+p} \leq x_{pq+1} + x_{pq + p - 1}\}.
\]

However, 
\[
	\mathbf{x} \notin P_{T(p,p,q)} \cup P_{T(p,q,p)} \cup P_{T(q,p,p)},
\]
which completes our proof.
\end{proof}

\begin{lemma}
\label{helperhelperlemma}
Let $q$ be an odd prime. For $a \in \ZZ/q\ZZ$, we have
\[
a\bigg\{\frac{q+1}{2},\dots,q-1\bigg\} = \bigg\{\frac{q+1}{2},\dots,q-1\bigg\} \Mod{q}
\]
if and only if $a \equiv 1 \Mod{q}$.
\end{lemma}
\begin{proof}
The statement 
\[
a\bigg\{\frac{q+1}{2},\dots,q-1\bigg\} = \bigg\{\frac{q+1}{2},\dots,q-1\bigg\} \Mod{q}
\]
is equivalent to the statement 
\[
a\bigg\{1,\dots,\frac{q-1}{2}\bigg \} = \bigg\{1,\dots,\frac{q-1}{2}\bigg \} \Mod{q}.
\]
If $q = 3$, it is clear that $a \equiv 1 \Mod{q}$; assume $q \neq 3$, and hence $q \geq 5$. Then as $a(q-1) \equiv -a \in \{\frac{q+1}{2},\dots,q-1\} \Mod{q}$, we must have $a \in \{1,\dots,\frac{q-1}{2}\} \Mod{q}$. If $a \not\equiv 1 \Mod{q}$, then there exists $b \in \{1,\dots,\frac{q-1}{2}\} \Mod{q}$ such that $ab \in \{\frac{q+1}{2},\dots,q-1\} \Mod{q}$, which is a contradiction.
\end{proof}

\begin{lemma}
\label{pqrhelperlemma}
Let $p$, $q$, and $r$ be odd prime numbers such that $p < q \leq r$. There exists an integer $m$ such that 
\[
	q \leq m \leq \lfloor qr/2 \rfloor,
\]
the addition $pm + pm$ overflows modulo $q$, and the addition $m + m$ does not overflow modulo $q$ or modulo $r$.
\end{lemma}
\begin{proof}
If $q = r$ then let
\[
	m = q + \bigg \lfloor \frac{q}{p} \bigg \rfloor.
\] 
Note that $m\perc q =  \lfloor q/p \rfloor$, as $0 \leq \lfloor q/p \rfloor < q$. As $0 \leq 2 \lfloor \frac{q}{p} \rfloor < 2q/p \leq q$, we have $(2m)\perc q = 2\lfloor q/p \rfloor$ and thus $m\perc q + m\perc q = (2m)\perc q$, so $m + m$ does not overflow modulo $q$.

On the other hand,  $pm = pq + p\lfloor q/p\rfloor$ and $0 \leq p\lfloor q/p\rfloor < q$, so $(pm)\perc q = p\lfloor q/p\rfloor$. Because $p < q$, we have $q\perc p \leq q/2$. Therefore,
\begin{align*}
(pm)\perc q + (pm)\perc q &= 2p\bigg \lfloor \frac{q}{p} \bigg \rfloor \\
&= 2p\bigg (\frac{q}{p}-\frac{q\perc p}{p}\bigg) \\
&\geq  2p\bigg (\frac{q}{p}-\frac{q}{2p}\bigg) \\
&= q,
\end{align*}
so the addition $pm + pm$ overflows modulo $q$.

Because $p^{-1} \neq 1 \Mod{q}$, by \cref{helperhelperlemma} the set
\[
	\bigg\{1,\dots,\frac{q-1}{2}\bigg\} \bigcap p^{-1}\bigg\{\frac{q+1}{2},\dots,q-1\bigg\} \Mod{q}
\]
is nonempty. Choose an element $\ell \in \ZZ/q\ZZ$ contained in the set above. Observe that 
\[
	q(r-1)/2 + (q-1)/2 = \bigg \lfloor \frac{qr}{2} \bigg \rfloor.
\]
and let
\[
	q \leq \ell_1,\dots, \ell_{\frac{r+1}{2}} 
\]
be the lifts of $\ell$ to $[q, \lfloor qr/2 \rfloor]$. Because $q \neq r$, the lifts $\ell_1,\dots, \ell_{\frac{r+1}{2}}$ all have distinct values modulo $r$ by the Chinese remainder theorem. Thus, there exists $\ell_k$ such that $\ell_k \in \{0,\dots, \frac{r-1}{2}\} \Mod{r}$. Set $m = \ell_k$.

To see that the addition $pm + pm$ overflows modulo $q$, notice that $m \in p^{-1}\{\frac{q+1}{2},\dots,q-1\} \Mod{q}$, so $(pm)\perc q \geq \frac{q+1}{2}$, and hence 
\[
	(pm)\perc q + (pm)\perc q \geq q+1.
\]

To see that the addition $m + m$ does not overflow modulo $q$ or $r$, observe that $m \in \{1,\dots,\frac{q-1}{2}\} \Mod{q}$, hence
\[
	m\perc q + m\perc q < q.
\] 
Similarly, since $m \in \{1,\dots,\frac{r-1}{2}\} \Mod{r}$, we have
\[
	m\perc r + m\perc r < r.
\] 
\end{proof}

\begin{proposition}
\label{pqr}
Let $n = pqr$ for primes $p$, $q$, and $r$ with $p < q \leq r$. Then there exists a realizable flag type $T$ such that 
\[
	P_T \not \subseteq P_{T(p,q,r)} \cup P_{T(p,r,q)} \cup P_{T(q,p,r)}\cup P_{T(q,r,p)} \cup P_{T(r,p,q)} \cup P_{T(r,q,p)}.
\]
\end{proposition}
\begin{proof}
Let $G = C_p \times C_q \times C_r$ and let $e_1, e_2, e_3$ be generators of each of the components. 

Define a sequence ${1=v_0,\dots,v_{n-1}}$ as follows. For $0 \leq i < p$, set $v_i = e_1^i$. For $p \leq i < n$ and $1 \leq i' < n$, write $i' = i'_1 + i'_2q + i_3'pq$ in mixed radix notation with respect to $(q,p,r)$. Inductively define $v_i$ as follows. Choose $i'$ minimal such that $e_2^{i'_1}e_1^{i'_2}e_3^{i'_3} \notin \{v_0,\dots,v_{i-1}\}$. Set $v_i = e_2^{i'_1}e_1^{i'_2}e_3^{i'_3}$. Observe that for $i \geq pq$, we have $i = i'$. Define a flag $\cF = \{F_i\}_{i \in [n]}$ by $F_i = \{v_0,\dots,v_i\}$.

By \cref{pqrhelperlemma}, there existss an integer $m$ such that 
\[
	q \leq m \leq \lfloor qr/2 \rfloor,
\]
the addition $pm + pm$ overflows modulo $q$, and the addition $m + m$ does not overflow modulo $q$ or modulo $r$. Moreover, 
\[
	2pm \leq 2p \lfloor qr/2 \rfloor < pqr.
\]

Write $pm = (pm)_1 + (pm)_2 q + (pm)_3 pq$ in mixed radix notation with respect to $(q,p,r)$. Then we have
\begin{align*}
F_1 &= \{1, e_1 \} \\
F_{p-1} &= \la e_1 \ra \\
F_{pm} &= \{e_2^{i_1}e_1^{i_2}e_3^{i_3} \; \mid \; i_1 + i_2q + i_3pq \leq pm, \; 0 \leq i_1 < q, \; 0 \leq i_2 < p, \; 0 \leq i_3\} \\
F_{2pm-1} &= \{e_2^{i_1}e_1^{i_2}e_3^{i_3} \; \mid \; i_1 + i_2q + i_3pq \leq 2pm-1, \; 0 \leq i_1 < q, \; 0 \leq i_2 < p, \; 0 \leq i_3\}.
\end{align*}
We have 
\[
F_1F_{p-1}  = F_{p-1}.
\]
Moreover, because the addition $pm + pm$ overflows modulo $q$, we have $(pm)_1 + (pm)_1 \geq q$. We have that
\begin{align*}
F_{pm}F_{pm} \subseteq \{e_2^{i_1}e_1^{i_2}e_3^{i_3} \; \mid \; &i_1 + i_2q + i_3pq \leq (q-1) + \min(p-1,2(pm)_2) q + 2(pm)_3 pq, \\
&\; 0 \leq i_1 < q, \; 0 \leq i_2 < p, \; 0 \leq i_3\}.
\end{align*}

Because $i_1 + i_2q + i_3pq \leq (q-1) + \min(p-1,2(pm)_2) q + 2(pm)_3 pq \leq 2pm-1$, we have
\[
	F_{pm}F_{pm} \subseteq F_{2pm-1}.
\]
We claim there exists $\mathbf{x} \in P_{T_{\cF}}$ such that $x_p > x_1 + x_{p-1}$ and $x_{2pm} > 2x_{pm}$. For $1 \leq i < p$ set $x_i = \frac{i}{2p}$. For $p \leq i < pq$ set $x_i = 1$. For $pq \leq i < pqr$ write $i = i_1 + i_2q + i_3pq$ in mixed radix notation with respect to $(q,p,r)$ and set $x_i = i_1 \frac{1}{4pq} + i_2 \frac{1}{2p} + i_3$. 

We first show that $0 \leq x_1 \leq \dots \leq x_{n-1}$. If $1 \leq i < p-1$, then
\[
	x_{i + 1} = \frac{i+1}{2p} \geq \frac{i}{2p} = x_i.
\]
Note also that
\[
	x_p = 1 \geq \frac{p-1}{2p} = x_{p-1}.
\]
If $p \leq i < pq-1$,
\[
	x_{i+1} = 1 = x_i.
\]
If $pq \leq i < n-1$ then write $i+1 = (i+1)_1 + (i+1)_2q + (i+1)pq$ in mixed radix notation with respect to $(q,p,r)$. If $i_1=q-1$ and $i_2 = p-1$ then $(i+1)_1 = (i+1)_2 = 0$ and $(i+1)_3 = i_3 + 1$. Then
\begin{align*}
x_{i+1} &= (i+1)_1 \frac{1}{4pq} + (i+1)_2 \frac{1}{2p} + (i+1)_3 \\
&= i_3 + 1 \\
&> i_1 \frac{1}{4pq} + i_2 \frac{1}{2p} + i_3 \\
&= x_{i}.
\end{align*}

Instead if $i_1=q-1$ and $i_2 \neq p-1$ then $(i+1)_1 = 0$ and $(i+1)_2=i_2+1$ and $(i+1)_3 = i_3$. Then
\begin{align*}
x_{i+1} &= (i+1)_1 \frac{1}{4pq} + (i+1)_2 \frac{1}{2p} + (i+1)_3 \\
&= (i_2+1) \frac{1}{2p} + i_3 \\
&> i_1 \frac{1}{4pq} + i_2 \frac{1}{2p} + i_3 \\
&= x_{i}.
\end{align*}

Finally, if $i \neq q-1$ then $(i+1)_1 = i_1+1$ and $(i+1)_2=i_2$ and $(i+1)_3 = i_3$. Then
\begin{align*}
x_{i+1} &= (i+1)_1 \frac{1}{4pq} + (i+1)_2 \frac{1}{2p} + (i+1)_3 \\
&= (i_1+1) \frac{1}{4pq} + i_2 \frac{1}{2p} + i_3 \\
&> i_1 \frac{1}{4pq} + i_2 \frac{1}{2p} + i_3 \\
&= x_{i}.
\end{align*}

Therefore, $0 \leq x_1 \leq \dots \leq x_{n-1}$.

We now show that for all $1 \leq i,j < n$, we have $x_{T_{\cF}(i,j)} \leq x_i + x_j$.

\noindent \textbf{Case $1$: $1 \leq i < p$ and $1 \leq j < p$}.
Then $k = \min(p-1, i+j)$. Then
\[
	x_k = \frac{k}{2p} \leq \frac{i}{2p} + \frac{j}{2p} = x_i + x_j.
\]
\noindent \textbf{Case $2$: $1 \leq i < p$ and $p \leq j < pq$}.
Then as $F_{pq-1} = \la e_1,e_2\ra$, we have $j < k < pq$. Thus
\[
	x_k = 1 \leq \frac{i}{2p} + 1 = x_i + x_j.
\]
\noindent \textbf{Case $3$: $1 \leq i < p$ and $pq \leq j < n$}.
Write $j = j_1 + j_2q + j_3pq$ in mixed radix notation with respect to $(q,p,r)$. Then $j < k \leq j_1 + \min(p-1,i+j_2)q + j_3pq$. Then
\begin{align*}
x_k &\leq x_{j_1 + \min(p-1,i+j_2)q + j_3pq} \\
&= j_1 \frac{1}{4pq} + \min(p-1,i+j_2) \frac{1}{2p} + j_3 \\
&\leq \frac{i}{2p} + j_1 \frac{1}{4pq} + j_2 \frac{1}{2p} + j_3 \\
&= x_i + x_j.
\end{align*}

\noindent \textbf{Case $4$: $p \leq i < pq$ and $p \leq j < pq$}.
Then as $F_{pq-1} = \la e_1,e_2\ra$, we have $j < k < pq$. Thus
\[
	x_k = 1 \leq \frac{i}{2q} + 1 = x_i + x_j.
\]
\noindent \textbf{Case $5$: $p \leq i < pq$ and $pq \leq j < n$}.
Write $j = j_1 + j_2q + j_3pq$ in mixed radix notation with respect to $(q,p,r)$. Because $F_{pq-1} = \la e_1,e_2\ra$ we have $j < k \leq (q-1) + (p-1)q + j_3pq$. Then
\begin{align*}
x_k &\leq x_{(q-1) + (p-1)q + j_3pq} \\
& = (q-1) \frac{1}{4pq} + (p-1) \frac{1}{2p} + j_3 \\
& \leq 1 + j_3 \\
&= 1 + j_1 \frac{1}{4pq} + j_2 \frac{1}{2p} + j_3  \\
&= x_i + x_j.
\end{align*}

\noindent \textbf{Case $6$: $pq \leq i < n$}.
Write 
\begin{align*}
i &= i_1 + i_2q + i_3pq \\
j &= j_1 + j_2q + j_3pq
\end{align*}
in mixed radix notation with respect to $(q,p,r)$. Recall that for $i \geq pq$, we have $v_i = e_2^{i_1}e_1^{i_2}e_3^{i_3}$. Thus, $j < k \leq \min(q-1,i_1+j_1) + \min(p-1,i_2+j_2)q + (i_3 + j_3)pq$.  Then
\begin{align*}
x_k &\leq x_{\min(q-1,i_1+j_1) + \min(p-1,i_2+j_2)p + (i_3 + j_3)pq} \\
& = \min(q-1,i_1+j_1) \frac{1}{4pq} + \min(p-1,i_2+j_2)\frac{1}{2p} + (i_3 + j_3) \\
&\leq (i_1 \frac{1}{4pq} + i_2 \frac{1}{2p} + i_3) + (j_1 \frac{1}{4pq} + j_2 \frac{1}{2p} + j_3) \\
&=x_i + x_j.
\end{align*}

Thus, for any integers $1 \leq i,j < n$, we have $x_{T_{\cF}(i,j)} \leq x_i + x_j$. Moreover, we have that
\[
	x_p = 1 > \frac{1}{2p} + \frac{p-1}{2p} = x_1 + x_{p-1}.
\]
Write 
\begin{align*}
pm &= (pm)_1 + (pm)_2q + (pm)_3pq \\
2pm &= (2pm)_1 + (2pm)_2q + (2pm)_3pq
\end{align*}
in mixed radix notation with respect to $(q,p,r)$ and recall that $pm + pm$ overflows modulo $q$. Therefore, either $(2pm)_3 = 2(pm)_3$ and $(2pm)_2 = 2(pm)_2 + 1$, or $(2pm)_3 = 2(pm)_3 + 1$. If $(2pm)_3 = 2(pm)_3$ and $(2pm)_2 > 2(pm)_2+1$ then
\begin{align*}
x_{2pm} &= (2pm)_1 \frac{1}{4pq} + (2pm)_2\frac{1}{2p} + (2pm)_3 \\
&\geq (2(pm)_2+1)\frac{1}{2p} + 2(pm)_3 \\
&> 2(pm)_1\frac{1}{4pq} + 2(pm)_2\frac{1}{2p} + 2(pm)_3 \\
&= 2x_{pm}.
\end{align*}
Otherwise, if $(2pm)_3 = 2(pm)_3 + 1$, then 
\begin{align*}
x_{2pm} &= (2pm)_1 \frac{1}{4pq} + (2pm)_2\frac{1}{2p} + (2pm)_3 \\
&\geq 2(pm)_3 + 1 \\
&> 2(pm)_1\frac{1}{4pq} + 2(pm)_2\frac{1}{2p} + 2(pm)_3 \\
&= 2x_{pm}.
\end{align*}

Further note that 
\[
	P_{T(q,p,r)} \cup P_{T(q,r,p)} \cup P_{T(r,q,p)} \cup P_{T(r,p,q)} \subseteq \{\mathbf{x} \in \RR^{n-1} \mid x_{p} \leq x_1 + x_{p-1} \}
\]
and
\[
	P_{T(p,q,r)} \cup P_{T(p,r,q)} \subseteq \{\mathbf{x} \in \RR^{n-1} \mid x_{2pm} \leq 2x_{pm}\}.
\]
We have $x_p > x_1 + x_{p-1}$ and $x_{2pm} > 2x_{pm}$, and thus our proof is complete. 
\end{proof}

\begin{proposition}
\label{4p}
Let $n = 4p$ for $p$ a prime not equal to $2$ or $3$. Then there exists a realizable flag type $T$ such that 
\[
	P_{T} \not \subseteq P_{T(2,2,p)} \cup P_{T(2,p,2)} \cup P_{T(p,2,2)}.
\]
\end{proposition}
\begin{proof}
Let $G = C_p \times C_2 \times C_2$ and let $e_1,e_2,e_3$ be generators of each of the components. Define a sequence $\{ v_0,\dots,v_{4p-1}\}$ as follows. Set 
\begin{align*}
v_0 &= 1 \\
v_1 &= e_1 \\
v_2 &= e_2 \\
v_3 &= e_2 e_1 \\
v_4 &= e_1^2 \\
v_5 &= e_2 e_1^2 \\
v_6 &= e_1^3 \\
v_7 &= e_2e_1^3.
\end{align*}
For $8 \leq i < n$ and $1 \leq i' < n$, write and $i' = i'_1 + i'_2p + i_3'2p$ in mixed radix notation with respect to $(p,2,2)$. Inductively define $v_i$ as follows. Choose $i'$ minimal such that $e_1^{i'_1}e_2^{i'_2}e_3^{i'_3} \notin \{v_0,\dots,v_{i-1}\}$. Set $v_i = e_1^{i'_1}e_2^{i'_2}e_3^{i'_3}$. Observe that for $i \geq 2p$, we have $i = i'$.

Note that 
\begin{align*}
F_1 &= \{ 1, e_1 \} \\
F_3 &= \la e_2\ra \{ 1,e_1 \} \\
F_5 &= \la e_2\ra\{ 1, e_1, e_1^2 \} \\
F_7 &= \la e_2\ra\{ 1, e_1, e_1^2, e_1^3 \} \\
F_{3p-1} &= \la e_1 \ra  \{ 1, e_2, e_3 \}.
\end{align*}
Therefore
\[
	F_1F_{3p-1} =  (\{ 1, e_1 \})( \la e_1\ra\{ 1, e_2, e_3 \}) = \la e_1\ra\{ 1, e_2, e_3 \} = F_{3p-1}
\]
\[
F_3F_3 = (\la e_2 \ra \{ 1,e_1 \})(\la e_2 \ra \{ 1,e_1 \}) = \la e_2 \ra\{ 1, e_1, e_1^2 \} = F_5
\]
\[
F_3F_5 =  (\la e_2 \ra \{ 1,e_1 \})(\la e_2 \ra\{ 1, e_1, e_1^2 \}) = \la e_2 \ra\{ 1, e_1, e_1^2,e_1^3 \}  = F_7.
\]

Suppose $p=5$. Then let $\mathbf{x} \in \RR^{19}$ be as follows:
\begin{align*}
x_1 &= 1 \\
x_2,x &= 1.4 \\
x_4,x_5 &= 2 \\
x_6,x_7 &= 3 \\
x_8,\dots,x_{14} &= 4 \\
x_{15},\dots,x_{19} &= 5.1.
\end{align*}
We claim that $\mathbf{x} \in P_{T_{\cF}}$ and $x_8 > x_5 + x_3$ and $x_{15} > x_1 + x_{14}$. It is clear that $0 \leq x_1 \leq \dots \leq x_{n-1}$.

We now show that for all $1\leq i,j < n$, we have $x_{T_{\cF}(i,j)} \leq x_i + x_j$. Let $k = T_{\cF}(i,j)$.

\noindent \textbf{Case 1: $i = 1$}. If $j = 1$ then $k = 4$ and 
\[
	x_4 = 2 \leq 2x_1.
\]
Observe that if $j = 2,3$ then $k = 4,5$, and
\[
	x_k = 2 \leq 1.4 + 1 \leq x_1 + x_j.
\]
If $j = 4,5$, then $k = 6,7$, and 
\[
	x_k = 3 \leq 2 + 1 = x_1 + x_j.
\]
If $j = 6,7$, then $k \leq 10$, and
\[
	x_k \leq 4 \leq 3 + 1 = x_1 + x_j
\].
If $8 \leq j < 15$, then as $v_1 = e_1$ and $F_{14} = \la e_1 \ra \{ 1, e_2, e_3 \}$, we have that $j < k < 14$. Thus
\[
	x_k \leq 4 \leq 4 + 1 = x_1 + x_j.
\] 
If $15\leq j < n$ then
\[
	x_k \leq 5.1 \leq 5.1 + 1 \leq x_1 + x_j.
\]

\noindent \textbf{Case 2: $i = 2,3$}.
If $j = 2,3$, then $k = 4,5$ so 
\[
	x_k = 2 \leq 1.4 + 1.4 = x_i + x_j.
\]
If $j = 4,5$ then $k = 6,7$ so
\[
	x_k = 3 \leq 2 + 1.4 = x_i + x_j.
\]
If $j = 6,7$ then $k < 15$ so
\[
	x_k \leq 4 \leq 3 + 1.4 \leq x_i + x_j.
\]
If $8 \leq j < n$ then
\[
	x_k \leq 5.1 \leq 4 + 1.4 \leq x_i + x_j.
\]

\noindent \textbf{Case 3: $i = 4,5$}
If $j < 10$ then because $F_9 = \la e_1,e_2 \ra$ we have $k < 10$. Thus
\[
	x_k \leq 4 \leq 2 + 2 \leq x_i + x_j. 
\]
If $10 \leq j < n$ then
\[
	x_k \leq 5.1 \leq 4 + 2 \leq x_i + x_j.
\]

\noindent \textbf{Case 4: $i \geq 6$}.
Then 
\[
	x_k \leq 5.1 \leq 3 + 3 \leq x_i + x_j.
\]

Thus $\mathbf{x} \in P_{T_{\cF}}$. One can see explicitly that $x_8 > x_5 + x_3$ and $x_{15} > x_1 + x_{14}$.
Because 
\[
	P_{T(2,2,5)} \cup P_{T(2,5,2)} \subset \{\mathbf{x} \in \RR^{19} \mid x_{15} \leq x_1 + x_{14} \}
\]
and
\[
	P_{T(5,2,2)} \subset \{\mathbf{x} \in \RR^{19} \mid x_8 \leq x_5 + x_3 \},
\]
our proof is complete for $p = 5$.

If $p \neq 5$, let $\mathbf{x} \in \RR^{4p-1}$ be as follows.
\begin{align*}
x_1 &= 1 \\
x_2, x_3 &= 1.4 \\
x_4, x_5 &= 2 \\
x_6, \dots, x_{3p-1} &= 2.9 \\
x_{3p}, \dots, x_{4p-1} &= 4.
\end{align*}
We claim that $\mathbf{x} \in P_{T_{\cF}}$ and $x_6 > x_3 + x_3$ and $x_{3p} > x_1 + x_{3p-1}$. It is clear that $0\leq x_1 \leq \dots \leq x_{n-1}$. 

We now show that for all integers $1 \leq i,j < n$, we have $x_{T(i,j)} \leq x_i + x_j$. Let $k = T(i,j)$.

\noindent \textbf{Case 1: $i = 1$}.
If $j = 1$, then $k = 4$ and
\[
	x_4 = 2 \leq 2x_1.
\]
Observe that if $j = 2,3$, then $k = 4,5$, and
\[
	x_k \leq x_5 = 2 \leq 1.4 + 1 \leq x_1 + x_j.
\]
If $j = 4,5$, then $k = 6,7$, and
\[
	x_k \leq x_5 = 2.9 \leq 2 + 1 \leq x_1 + x_j.
\]
If $6 \leq j < 3p$, then as $v_1 = e_1$ and $F_{3p-1} = \la e_1 \ra \{ 1, e_2, e_3 \}$, we have $j < k < 3p$. Thus
\[
	x_k = 2.9 \leq 2.9 + 1 = x_j + x_1.
\]
If $j \geq 3p$, then
\[
	x_k = 4 \leq 4 + 1 \leq x_1 + x_j.
\]
\textbf{Case 2: $i = 2,3$}. 
If $j = 2,3$, then $k \leq 5$ and
\[
	x_k = 2 \leq 1.4 + 1.4 = x_i + x_j.
\]
If $j = 4,5$ then $k \leq 7$ and
\[
	x_k \leq 2.9 \leq 2 + 1.4 = x_i + x_j.
\]
If $j \geq 6$, then
\[
	x_k \leq 4 \leq 2.9 + 1.4 = x_i + x_j.
\]
\textbf{Case 3: $i \geq 4$}.
Then we have
\[
	x_k \leq 4 \leq 2 + 2 \leq x_i + x_j.
\]

Therefore, for any integers $1 < i,j < n$,  we have $x_{T_{\cF}(i,j)} \leq x_i + x_j$, so $x \in P_{T_{\cF}}$. Moreover, it is clear that $x_6 > x_3 + x_3$ and $x_{3p} > x_1 + x_{3p-1}$.

Again, we have 
\[
	P_{T(2,2,p)} \cup P_{T(2,p,2)} \subset \{\mathbf{x} \in \RR^{4p-1} \mid x_{3p} \leq x_1 + x_{3p-1} \}
\]
and because $p \geq 7$, we have
\[
	P_{T(p,2,2)} \subset \{\mathbf{x} \in \RR^{4p-1} \mid x_6 \leq x_3 + x_3 \},
\]
and thus our proof is complete.
\end{proof}

\section{Tower types}
Our aim in this section will be to prove the following theorem.

\lenstrathm*
\begin{proof}

If $n = p^k$ is a prime power, \cref{minimalitycompletethm} proves that $T(p,\dots,p)$ is the unique flag type that is minimal among flags of tower type $\mfT$.

If $n = 4$ and $\mfT = (4)$ apply \cref{LC4}. If $n = 6$, apply \cref{LC6}. Every composite $n < 8$ is either prime, $4$, or $6$ the theorem is proven for $n < 8$. If $n = 8$, apply \cref{casesfor8corol}.
If $n = 2p$ for an odd prime $p$ and  $\mfT = (2,p)$ then apply \cref{LC2p}. Similarly, if $n = 3p$ for a prime $p$ and $\mfT = (3,p)$ then apply \cref{LC3p}.
\end{proof}

We will need the following lemma.

\begin{lemma}
\label{mleqilemma}
Let $\cF$ be a flag of tower type $\mfT = (n_1,\dots,n_t)$. Suppose we have $i,j,i+j \in [n]$ such that $i + j$ does not overflow modulo $\mfT$ and $T_{\cF}(i,j) < i + j$. Then there exists an integer $1 < m < n$ such that $m \mid n$ and $i + j$ overflows modulo $m$ and $m < i$.
\end{lemma}
\begin{proof}
We prove the lemma when $\cF$ is a flag of an abelian group. The proof in the case of field extensions is identical. If $\cF$ is a flag of a group, let $m = \lv \Stab(F_iF_j)\rv$. Clearly $m \mid n$. By \cref{overflowlemma}, we have $1 < m < n$. 

Write $i$, $j$, and $i + j$ in mixed radix notation with respect to $(m, n/m)$ as
\begin{align*}
i &= i_1 + i_2 m \\
j &= j_1 + j_2 m  \\
i+j &= (i+j)_1 + (i+j)_2 m.
\end{align*}
By \cref{overflowlemma},
\begin{align*}
\lv \Stab(F_iF_j)F_i \rv &= (i_2 + 1)m \\
\lv \Stab(F_iF_j)F_j \rv &= (j_2 + 1)m \\
\lv F_iF_j \rv &= (i_2 + j_2 + 1)m. 
\end{align*}
Because $i + j$ overflows modulo $m$, we have $i \neq m$. Assume for the sake of contradiction that $i < m$; because $\lv \Stab(F_iF_j)F_i \rv = m$, we have that $F_i$ is contained in the group $\Stab(F_iF_j)$. Therefore the group generated by $F_i$ is contained in the group $\Stab(F_iF_j)$. The group generated by $F_i$ has cardinality $n_1\dots n_k$ for some $1 \leq k \leq t$; thus $n_1 \dots n_k \mid m$.

Because $i + j$ overflows modulo $m$ and does not overflow modulo $\mfT = (n_1,\dots,n_t)$, we have $n_1\dots n_k \neq m$. Write $i$, $j$, and $i+j$ in mixed radix notation with respect to $(n_1\dots n_k, m/(n_1\dots n_k), n/m)$ as
\begin{align*}
i &= i'_1 + i'_2 (n_1\dots n_k) + i'_3 m \\
j &= j'_1 + j'_2 (n_1\dots n_k) + j'_3 m \\
i+j &= (i+j)'_1 + (i+j)'_2 (n_1\dots n_k) + (i+j)'_3 m.
\end{align*}

\noindent Because the addition $i + j$ does not overflow modulo $\mfT$, we must have $i'_1 + j'_1 = (i+j)'_1$. Because $\la F_i \ra = n_1\dots n_k$, we have $i < n_1\dots n_k$, and thus $i_2' = 0$ so $i'_2 +  j'_2 = (i+j)'_2$.  Because $i'_1 + i'_2 (n_1\dots n_k) = i_1$, $j'_1 + j'_2 (n_1\dots n_k) = j_1$, and $(i+j)'_1 + (i+j)'_2 (n_1\dots n_k) = (i+j)_1$, we have that $i_1 + j_1 = (i+j)_1$; this is a contradiction, as the addition $i + j$ overflows modulo $m$.
\end{proof}

The rest of this section is dedicated to proving \cref{LC4}, \cref{LC6}, \cref{casesfor8corol}, \cref{LC2p}, and \cref{LC3p}.

\begin{corollary}
\label{LCfori12}
Let $i = 1$ or $i = 2$ and choose $j$ such that $i \leq j < i + j < n$. Let $\cF$ be a flag with tower type $\mfT$. If $T_{\cF}(i,j) < i+j$
then $i + j$ overflows modulo $\mfT$.
\end{corollary}
\begin{proof}
If $i + j$ does not overflow modulo $\mfT$ and $T_{\cF}(i,j) < i+j$ apply \cref{mleqilemma} and observe that there must exist an integer $1 < m < i \leq 2$, which is a contradiction. 
\end{proof}

\begin{corollary}
\label{nminus1corollary}
If $i + j = n-1$ then for any flag $\cF$ we have $T_{\cF}(i,j) = n-1$.
\end{corollary}
\begin{proof}
Observe that $i+j$ does not overflow modulo $m$ for any $m \mid n$, and apply \cref{overflowlemma}.
\end{proof}

\begin{corollary}
\label{LC4}
We have that $T(4)$ is the unique flag type minimal among flags of tower type $\mfT = (4)$.
\end{corollary}
\begin{proof}
Suppose for the sake of contradiction that $\cF$ be a flag type of tower type $\mfT = (4)$ such that $T_{\cF} \not \geq T(4)$.

By \cref{testflagineqlemma}, there exists a corner $(i,j)$ of $T(4)$ such that $T_{\cF}(i,j) < i+j$. Without loss of generality, we may assume $i \leq j$; this implies that $i = 1$. By \cref{LCfori12}, this implies $i+j$ overflows modulo $\mfT = (4)$, which is a contradiction.
\end{proof}

\begin{corollary}
\label{LC6}
We have that $T(6)$ is the unique flag type minimal among flags of tower type $\mfT = (6)$.
\end{corollary}
\begin{proof}
Suppose for the sake of contradiction that $\cF$ be a flag type of tower type $\mfT = (6)$ such that $T_{\cF} \not \geq T(6)$.

By \cref{testflagineqlemma}, there exists $(i,j)$ such that $T_{\cF}(i,j) < i+j$. Without loss of generality, we may assume $i \leq j$. If $i + j = 5$, then \cref{nminus1corollary} implies that $T_{\cF}(i,j) = i+j$, which is a contradiction. Therefore $i = 1$ or $i = 2$. By \cref{LCfori12}, this implies $i+j$ overflows modulo $\mfT = (6)$, which is a contradiction. 
\end{proof}

\begin{corollary}
\label{casesfor8corol}
We have that $T(2,4)$ is the unique flag type minimal among flags of tower type $\mfT = (2,4)$. We have that $T(4,2)$ is the unique flag type minimal among flags of tower type $\mfT = (4,2)$.
\end{corollary}
\begin{proof}
Suppose for the sake of contradiction that $\cF$ be a flag type of tower type $\mfT = (2,4)$ such that $T_{\cF} \not \geq T(2,4)$.

By \cref{testflagineqlemma}, there exists $(i,j)$ of $T(2,4)$ such that $T_{\cF}(i,j) < i+j$ and $i+j$ does not overflow modulo $(2,4)$. If $i + j = 7$, then \cref{nminus1corollary} implies that $T_{\cF}(i,j) = i+j$, which is a contradiction. Therefore $i = 1$ or $i = 2$. By \cref{LCfori12}, this implies $i+j$ overflows modulo $\mfT = (2,4)$, which is a contradiction. 

The proof for $\mfT = (4,2)$ is identical. 
\end{proof}

\begin{corollary}
\label{LC2p}
For any prime $p$, we have that $T(2,p)$ is the unique flag type minimal among flags of tower type $\mfT = (2,p)$.
\end{corollary}
\begin{proof}
Suppose for the sake of contradiction that $\cF$ be a flag type of tower type $\mfT = (2,p)$ such that $T_{\cF} \not \geq T(2,p)$.

By \cref{testflagineqlemma}, there exists $(i,j)$ of $T(2,p)$ such that $T_{\cF}(i,j) < i+j$ and $i+j$ does not overflow modulo $(2,p)$. By \cref{overflowlemma}, we have that $i + j$ overflows modulo $p$. Without loss of generality we may suppose $i \leq j$. By \cref{mleqilemma}, we have $p < i$; thus $i + j > 2p$, which is a contradiction.
\end{proof}

\begin{corollary}
\label{LC3p}
For any prime $p$, we have that $T(3,p)$ is the unique flag type minimal among flags of tower type $\mfT = (3,p)$.
\end{corollary}
\begin{proof}
Suppose for the sake of contradiction that $\cF$ be a flag type of tower type $\mfT = (3,p)$ such that $T_{\cF} \not \geq T(3,p)$.

By \cref{testflagineqlemma}, there exists $(i,j)$ of $T(3,p)$ such that $T_{\cF}(i,j) < i+j$ and $i+j$ does not overflow modulo $(3,p)$. By \cref{overflowlemma}, we have that $i + j$ overflows modulo $p$. Without loss of generality we may suppose $i \leq j$. By \cref{mleqilemma}, we have $p < i$. Because $p < i \leq j$ and $i + j$ overflows modulo $p$, we have $i + j \geq 3p$, which is a contradiction.
\end{proof}

\section{The structure of corners}
\label{cornerstructuresec}

\cornersineq*
\begin{proof}
We prove the theorem for flag types realizable for groups; the proof for flag types realizable by field extensions is identical.

Let $(i,j)$ be a corner of $T$ and suppose for the sake of contradiction that $T(i,j) < i + j$.
\[
	F_iF_{j-1} = F_iF_j
\]
which would imply that $(i,j)$ is not a corner.

To this end, set $M = \Stab(F_iF_j)$ and $m = \lv\Stab(M) \rv$. Note that $m > 1$. Write $i$ and $j$ in mixed radix notation with respect to $(m,n/m)$ as $i = \alpha_1 + \alpha_2m$ and $j = \beta_1 + \beta_2m$ and note that $\alpha_1 + \beta_1 > m$ (because $k \neq i + j$). Thus, $\beta_1 \neq 1$. 
 
Note that $m^{-1} \lv MF_{j-1}\rv \geq \lceil j/m \rceil = \beta_2 + 1$ and $m^{-1} \lv MF_{j} \rv = \beta_2 + 1$, so
\[
	MF_{j-1} = MF_j.
\]
Note that $\Stab(MF_{i}MF_{j-1}) = M$. Define
\[
	M' = \Stab(F_{i}F_{j-1})
\]
and set $m' = \lv M' \rv$. Note that because 
\[
	\Stab(F_{i}F_{j-1}) \subseteq \Stab(MF_{i}MF_{j-1}),
\]
we have $M' \subseteq M$ and thus $m' \mid m$.

First suppose $\alpha_1 + \beta_1 - 1 - m < m'$. By \cref{kneserorig}, we have
\[
	\lv F_{i}F_{j-1}\rv = (\alpha_1 + \beta_1 + 1)m
\]
and thus 
\[
	F_{i}F_{j-1} = F_{i}F_{j}.
\]

Now suppose $\alpha_1 + \beta_1 - 1 - m \geq m'$. If $M' \neq M$, then
\[
	\lv F_{i}F_{j-1}\rv > \lv F_{i} F_{j}\rv,
\]
which is a contradiction. Thus we must have $M' = M$ so by \cref{kneserorig}, we have
\[
	\lv F_{i}F_{j-1}\rv = (\alpha_1 + \beta_1 + 1)m
\]
and thus $F_{i}F_{j-1} = F_{i}F_{j}$. Thus, we have a contradiction so $k \geq i + j$.
\end{proof}

\begin{lemma}
\label{existsapt18lemma}
There exists flag type $T$, realizable for groups and field extensions, and a point $\mathbf{x} \in P_T$ such that $x_2 > 2x_1$ and $x_{15} > x_8 + x_7$.
\end{lemma}
\begin{proof}
Choose the $\mathbf{x}$ given in the proof of \cref{pqr} for $n = 18$.
\end{proof}

The following lemma is stated for flags of field extensions; the analogous statement for abelian groups is true, and has an identical proof. 
\begin{lemma}
\label{strucof18lemma}
Let $\cF$ be any flag of a degree $18$ field extension $L/K$ such that  $T_{\cF}(1,1) = 1$ and $T_{\cF}(7,8) \leq 14$.
Then there exists $\alpha,\beta,\gamma \in L$ such that 
\begin{enumerate}
\item $[K(\alpha) \colon K] = 2$;
\item $[K(\beta) \colon K] = 3$;
\item $K(\alpha, \beta, \gamma) = L$;
\item $F_1 = K(\alpha)$;
\item $F_8 = K(\beta)\la 1, \alpha, \gamma \ra$;
\item $K(\beta)F_7 = K(\beta)\la 1, \alpha, \gamma \ra$;
\item and $F_{14} = K(\beta)\la 1, \alpha, \gamma , \alpha\gamma, \gamma^2\ra $.
\end{enumerate}
\end{lemma}
\begin{proof}
Let $\{1=v_0,\dots,v_{n-1}\}$ be a $K$-basis of $L$ such that $F_i = K\la v_0,\dots,v_i\ra$. Let $\alpha \coloneqq v_1$. As $v_1^2 \in K\la 1,v_1 \ra$ and $\alpha \notin K$, we have that $\deg(\alpha) = 2$.

\cref{overflowlemma} implies that $[\Stab(F_7F_8) \colon K] = 3$. Let $\beta$ be any primitive element of the cubic extension $\Stab(F_7F_8)$; such an element exists because any extension of prime degree is simple. Moreover \cref{overflowlemma} implies that $K(\beta)F_8 = F_8$. Let $\gamma \in F_8 \setminus K(\alpha,\beta)$. Thus, $K(\alpha, \beta, \gamma) = L$.

Thus, $F_8 = K(\beta)\la 1, \alpha, \gamma \ra$. \cref{overflowlemma} implies that $K(\beta) F_7 = K(\beta)\la 1, \alpha, \gamma \ra$.
By \cref{overflowlemma}, $F_7 F_8 = F_{14}$ so $F_{14} = K(\beta)\la 1, \alpha, \gamma , \alpha\gamma, \gamma^2\ra$.
\end{proof}

\cornersstrongineq*
\begin{proof}
We will prove that there is such a $T$ that is minimal for field extensions; the analogous statement for abelian groups is true, and has an identical proof. The two statements combined prove the theorem.
 
By \cref{existsapt18lemma}, there exists some flag type $T$ that is realizable for field exensions and $\mathbf{x}\in P_T$ such that $x_2 > 2x_1$ and $x_{15} > x_8 + x_7$.

Let $\cF$ be a flag of a field extension $L/K$ realizing $T$; then $T_{\cF}(1,1)=1$ and $T_{\cF}(7,8) \leq 14$. 
 
\noindent Apply \cref{strucof18lemma} to obtain $\alpha,\beta,\gamma \in L$ such that 
\begin{enumerate}
\item $[K(\alpha) \colon K] = 2$;
\item $[K(\beta) \colon K] = 3$;
\item $K(\alpha, \beta, \gamma) = L$;
\item $F_1 = K(\alpha)$;
\item $F_8 = K(\beta)\la 1, \alpha, \gamma \ra$;
\item $K(\beta)F_7 = K(\beta)\la 1, \alpha, \gamma \ra$;
\item and $F_{14} = K(\beta)\la 1, \alpha, \gamma , \alpha\gamma, \gamma^2\ra $.
\end{enumerate}

We will prove that there exists $j \leq 12$ and $k \geq 15$ such that $T(1,j) = k$. Suppose for the sake of contradiction that the claim is false. Then
\[
	K(\alpha)F_{12} \subseteq F_{14} = K(\beta)\la 1,\alpha,\gamma,\alpha\gamma,\gamma^2\ra.
\]
\noindent Write $\alpha^{-1} = b\alpha + c$ for $b,c \in K$ and $b \neq 0$. Then, 
\begin{align*}
F_{12} &\subseteq K(\beta)\la 1,\alpha,\gamma,\alpha\gamma,\gamma^2\ra \cap \alpha^{-1}K(\beta)\la 1,\alpha,\gamma,\alpha\gamma,\gamma^2\ra \\
&= K(\beta)\la 1,\alpha,\gamma,\alpha\gamma,\gamma^2\ra \cap K(\beta)\la \alpha^{-1},1,\alpha^{-1}\gamma,\gamma,\alpha^{-1}\gamma^2\ra \\
&= K(\beta)\la 1,\alpha,\gamma,\alpha\gamma,\gamma^2\ra \cap K(\beta)\la b\alpha+c,1,b\alpha\gamma + c\gamma,\gamma,b\alpha\gamma^2 + c\gamma^2\ra \\
&= K(\beta)\la 1,\alpha,\gamma,\alpha\gamma,\gamma^2\ra \cap K(\beta)\la 1, \alpha,\gamma,\alpha\gamma, b\alpha\gamma^2 + c\gamma^2\ra \\
&= K(\beta)\la 1,\alpha,\gamma,\alpha\gamma \ra,
\end{align*}
but $\dim_{K}F_{12} = 13$ and $\dim_{K}K(\beta)\la 1,\alpha,\gamma,\alpha\gamma \ra = 12$, which is a contradiction.
\end{proof}

\end{document}